\documentclass[AMA,STIX1COL]{WileyNJD-v2}
\articletype{Research article}%
\usepackage{lineno}
\usepackage{graphicx} 
\usepackage{subcaption}
\usepackage{algorithm}% <=================http://ctan.org/pkg/algorithms
\usepackage{algpseudocode}% <========== http://ctan.org/pkg/algorithmicx
\usepackage{cleveref}

\newtheorem{example}{Problem}[section]

\begin{document}

\title{Composite Anderson acceleration method with dynamic window-sizes and optimized damping}

\author[1,2]{Kewang Chen* }

\author[2]{Cornelis Vuik}

% \author[3]{Author Three}

\authormark{K. Chen and C. Vuik}

\address[1]{\orgdiv{College of Mathematics and Statistics}, \orgname{Nanjing University of Information Science and Technology}, \orgaddress{\state{Nanjing, 210044}, \country{China}}}

\address[2]{\orgdiv{Delft Institute of Applied Mathematics}, \orgname{Delft University of Technology}, \orgaddress{\state{Delft, 2628CD}, \country{the Netherlands}}}

% \address[3]{\orgdiv{Org Division}, \orgname{Org Name}, \orgaddress{\state{State name}, \country{Country name}}}

\corres{*Corresponding author.\\ \email{kwchen@nuist.edu.cn}}

%\presentaddress{This is sample for present address text this is sample for present address text}

\abstract[Summary]{In this paper, we propose and analyze a set of fully non-stationary Anderson acceleration algorithms with dynamic window sizes and optimized damping. Although Anderson acceleration (AA) has been used for decades to speed up nonlinear solvers in many applications, most authors are simply using and analyzing the stationary version of Anderson acceleration (sAA) with fixed window size and a constant damping factor. The behavior and potential of the non-stationary version of Anderson acceleration methods remain an open question. Since most efficient linear solvers use composable algorithmic components. Similar ideas can be used for AA to solve nonlinear systems. Thus in the present work, to develop non-stationary Anderson acceleration algorithms, we first propose two systematic ways to dynamically alternate the window size $m$ by composition. One simple way to package sAA(m) with sAA(n) in each iteration is applying sAA(m) and sAA(n) separately and then average their results. It is an additive composite combination. The other more important way is the multiplicative composite combination, which means we apply sAA(m) in the outer loop and apply sAA(n) in the inner loop. By doing this, significant gains can be achieved. Secondly, to make AA to be a fully non-stationary algorithm, we need to combine these strategies with our recent work on the non-stationary Anderson acceleration algorithm with optimized damping (AAoptD), which is another important direction of producing non-stationary AA and nice performance gains have been observed. Moreover, we also investigate the rate of convergence of these non-stationary AA methods under suitable assumptions. Finally, our numerical results show that some of these proposed non-stationary Anderson acceleration algorithms converge faster than the stationary sAA method and they may significantly reduce the storage and time to find the solution in many cases. }


%\jnlcitation{\cname{%
%\author{K. Chen} and
%\author{C. Vuik}} (\cyear{2022}), 
%\ctitle{A regime analysis of Atlantic winter jet variability applied to evaluate HadGEM3-GC2}, \cjournal{Q.J.R. Meteorol. Soc.}, \cvol{2017;00:1--6}.}

\maketitle

%\footnotetext{\textbf{Abbreviations:} AA, Anderson acceleration; sAA, stationary Anderson acceleration; AAoptD, Anderson acceleration algorithm with optimized damping.}

\section{Introduction} \label{sec:intro}
In this part, we first present a literature review on the development of Anderson acceleration method and its applications. Then we discuss our main motivations and the structure for the present work. In 1962, Anderson \cite{An1965} developed a technique for accelerating the convergence of the Picard iteration associated with a fixed-point problem which is called Extrapolation Algorithm. Since that, this technique has enjoyed remarkable success and wide usage, especially in electronic structure computations, where it is known as Anderson mixing. The technique is now called Anderson acceleration (AA) in the applied mathematics community. In contrast to Picard iteration, which uses only one previous iterate, AA method proceeds by linearly recombining a list of previous iterates in a way such that approximately minimizes the linearized fixed-point residual. The usual general form of Anderson acceleration with damping is given in \cref{alg:Anderson}.

\begin{algorithm}
\caption{Anderson acceleration: $\mathbf{AA(m)}$}
\label{alg:Anderson}
\begin{algorithmic}
\State{Given: $x_0$ and $m\geq 1$.}
\State{Set: $x_{1}=g(x_0).$}
\For{$k=0,1,2,\cdots $}
\State{Set: $m_{k}=\min\{m,k\}$.}
\State{Set: $F_k=(f_{k-m_k},\cdots,f_k)$, where $f_i=g(x_i)-x_i$.}
\State{Determine: $\alpha^{(k)}=\left(\alpha_0^{(k)},\cdots,\alpha_{m_k}^{(k)}\right)^{T}$ that solves }
\State{$\ \ \ \ \ \ \ \ \ \ \ \displaystyle \min_{\alpha=(\alpha_0,\cdots,\alpha_{m_k})^{T}}\|F_k\alpha\|_2$ $\ s.\ t.$ $\displaystyle\sum_{i=0}^{m_k}\alpha_i=1.$}
\State{Set: $\displaystyle x_{k+1}=(1-\beta_k)\sum_{i=0}^{m_k}\alpha_{i}^{(k)}x_{k-m_k+i}+\beta_k\sum_{i=0}^{m_k}\alpha_{i}^{(k)}g(x_{k-m_k+i})$.}
\EndFor
\end{algorithmic}
\end{algorithm}
Here, $f_k$ is the residual for the $k$th iteration; $m$ is the window size which indicates how many history residuals will be used in the algorithm. It is usually a fixed number during the procedure. The value of $m$ is typically no larger than $3$ in the early days of applications and now this value could be as large as up to 100.\cite{An2019} $\beta_k\in (0,1]$ is a damping factor (or a relaxation parameter) at $k$th iteration. We have, for a fixed window size $m$:
\begin{equation*}
   \beta_k=
    \begin{cases}
      1,  & \text{no damping,}\\
      \beta, \  (\text{a constant independent of $k$}) & \text{stationary Anderson acceleration,}\\
      \beta_k,\ (\text{depending on $k$}) \  & \text{non-stationary Anderson acceleration.}
    \end{cases}       
\end{equation*}
We can also formulate the above constrained optimization problem as an equivalent unconstrained least-squares problem \cite{ToKe2015,wa2011}:
\begin{equation}
\displaystyle \min_{(\omega_1,\cdots,\omega_{m_k})^{T}}\left\|f_k+\sum_{i=1}^{m_k}\omega_i(f_{k-i}-f_k)\right\|_2
\end{equation}
One can easily recover the original problem by setting
$$\omega_0=1-\sum_{i=1}^{m_k}\omega_i.$$
%This formulation of the linear least-squares problem is not optimal for implementation, we will discuss this in more detail in \cref{sec:proof}.

Anderson acceleration methods are ``essentially equivalent'' to the nonlinear GMRES methods \cite{CaMi1998,Mi2005,OoWa2000,WaOo1997,WaNi2011} and the direct inversion on the iterative subspace method (DIIS).\cite{LiYa2013,Pu1980,Pu1982} It is also in a broad category with methods based on quasi-Newton updating.\cite{EiNe1987,Ey1996,FaSa2009,Ha2010,Ya2009} For example, Walker and Ni \cite{WaNi2011} showed that AA without truncation is equivalent in a certain sense to the GMRES method on linear problems and Fang and Saad \cite{FaSa2009} had clarified a remarkable relationship of AA to quasi-Newton methods. However, unlike Newton-like methods, one advantage of AA is that it does not require the expensive computation or approximation of Jacobians or Jacobian-vector products. Although AA has been used for decades, convergence analysis has been reported only recently. For the linear case, Toth and Kelley \cite{ToKe2015} first proved the stationary version of AA (sAA) without damping is locally r-linearly convergent if the fixed point map is a contraction and the coefficients in the linear combination remain bounded. Later, Evens et al. \cite{Ev2020} extended the result to AA with damping factors and proved the convergence rate is $\theta_k((1-\beta_{k-1})+\beta_{k-1}\kappa)$, where $\kappa$ is the Lipschitz constant for the function $g(x)$ and $\theta_k$ is the ratio quantifying the convergence gain provided by AA in step $k$. Recently, in 2019, Pollock et al. \cite{Po2019} applied sAA to the Picard iteration for solving steady incompressible Navier–Stokes equations and proved that the acceleration improves the convergence rate of the Picard iteration. More recently, De Sterck \cite{DeHe2021} extended the result to more general fixed-point iteration $x=g(x)$, given knowledge of the spectrum of $g'(x)$ at fixed-point $x^*$ and Wang et al. \cite{WaSt2021} extended the result to study the asymptotic linear convergence speed of sAA applied to Alternating Direction Method of Multipliers (ADMM) method. Sharper local convergence results of AA remain a hot research topic in this area. For more related results about Anderson acceleration and its applications, we refer the interested readers to papers\cite{BiKe2021,Br2015,YuLi2018,WeGa2019,To2017,Ya2021,Zh2020} and references therein. 

Although AA has been widely used and studied for decades, most authors are simply using and analyzing the stationary version of Anderson acceleration with fixed window size and a constant damping factor. The behavior and potential of the non-stationary versions of the Anderson acceleration method have not been deeply studied and few results have been reported. Besides, the early days of the Anderson Mixing method (the 1980s, for electronic structure calculations) initially dictated the window size $m\leq 3$ due to the storage limitations and costly $g$ evaluations. However, in recent years and a broad range of contexts, the window size $m$ ranging from $20$ to $100$ has also been considered by many authors. For example, Walker and Ni \cite{WaNi2011} used $m=50$ to solve the nonlinear Bratu problem. A natural question is that should we try to further speed up Anderson acceleration method or try to use a larger size of the window? No such comparison results have been reported. As we know, there are two main possible directions for producing non-stationary AA. One is choosing different damping factors $\beta_k$ in each iteration, see our recent work on the non-stationary Anderson acceleration algorithm with optimized damping (AAoptD).\cite{Chen2022} The other way of making AA to be a non-stationary algorithm is to alternate the window size during iterations. But no systematic ways have been proposed to dynamically alternate the window size $m$. Since most efficient linear solvers use composable algorithmic components,\cite{Brown2012,Kirby2018} similar ideas can be used for AA(m) and AA(n) to solve nonlinear systems. Moreover, the combination of choosing optimized damping factors and alternating window sizes may lead to significant gains over the stationary AA. Therefore, in the present work, we propose and study the fully non-stationary Anderson acceleration algorithms with dynamic window sizes and optimized damping. 

% The outline is not required, but we show an example here.
The paper is organized as follows. Our new algorithms and analysis are developed in
\Cref{sec:main}; the rate of convergence to those algorithms are studied in \Cref{sec:theory}; experimental results and discussion are in \Cref{sec:experiments}; conclusions follow in \Cref{sec:conclusions}.

\section{Fully non-stationary Anderson acceleration}\label{sec:main}
\subsection{Non-stationary AA with dynamic window sizes} 
\label{subsec:main-1}
In this part, we propose two systematical ways to dynamically alternate the window size $m$ in sAA. The natural idea is to package $sAA(m)$ with $sAA(n)$ in each iteration. To begin with, an easy additive composite combination may be written as:
$$x_{k+1}=\alpha_m\times sAA(x_{k},\cdots,x_{k-m},m)+\alpha_n \times sAA(x_k,\cdots,x_{k-n},n),$$ 
where $\alpha_m+\alpha_n=1$. For example, one can take $\alpha_m=1/2,\alpha_n=1/2$. We obtain the following algorithm $AA(m+n)$ as in \Cref{alg:Anderson_Anderson_sum}.

\begin{algorithm}
\caption{Anderson acceleration with dynamic window sizes: $\mathbf{AA(m+n)}$}
\label{alg:Anderson_Anderson_sum}
\begin{algorithmic}
\State{Given: $x_0$, $\alpha_m$, $\alpha_n$ and $m,n$.}
\State{Set: $x_{1}=g(x_0).$}
\For{$k=0,1,2,\cdots, $}
\State{\textbf{Set:} $\mathbf{x_{k+1/2}^m\leftarrow}$ \textbf{applying} $\mathbf{sAA(m)}$ \textbf{as given in \cref{alg:Anderson}.}}
\State{\textbf{Set:} $\mathbf{x_{k+1/2}^n\leftarrow}$ \textbf{applying} $\mathbf{sAA(n)}$ \textbf{as given in \cref{alg:Anderson}.}}
\State{\textbf{Set:} $\mathbf{x_{k+1}=\alpha_m x_{k+1/2}^m+\alpha_n x_{k+1/2}^n}$.}
\EndFor
\end{algorithmic}
\end{algorithm}
Without loss of generality, we assume here that $m>n$, so the total amount of work is less than that of applying $AA(m)$ twice. For example, we can choose window size $n=0$ (the Picard iteration) or $n=1$. To compare the performance of $AA(m+n)$ with regular stationary $AA(m)$, we can compare it with applying $AA(m)$ twice or compare the total time needed for each method to find the solution. 

The other more important way is the multiplicative composite combination, which means we apply sAA(m) first in the outer loop and then apply sAA(n) in the inner loop. The idea is similar to the hybrid method GMRESR (GMRES Recursive), which is proposed by Van der Vorst and Vuik\cite{vvuik94} and further investigated by Vuik\cite{vuik93}. We start with composite sAA(m) with $sAA(0)$ (i.e. Picard iteration) in each iteration. This means that after applying $sAA(m)$ without damping, we get,
$$\displaystyle x_{k+1/2}=\sum_{i=0}^{m_k}\alpha_{i}^{(k)}g(x_{k-m_k+i}).$$
Then, we take the result $x_{k+1/2}$ as an input and apply Picard iteration. That is,{}
$$x_{k+1}=g(x_{k+1/2}).$$
Putting these together, we have the following non-stationary algorithm $AA(m,AA(0))$ as in \Cref{alg:Anderson_picard_mul}.
\begin{algorithm}
\caption{Anderson acceleration with dynamic window-sizes: $\mathbf{AA(m,AA(0))}$}
\label{alg:Anderson_picard_mul}
\begin{algorithmic}
\State{Given: $x_0$ $iterM$, $iterN$ and $m\geq 1$.}
\State{Set: $x_{1}=g(x_0).$}
\For{$k=0,1,2,\cdots,iterM $}
\State{\textbf{Set:} $\mathbf{x_{k+1/2}^m\leftarrow}$ \textbf{applying} $\mathbf{sAA(m)}$ \textbf{as given in \Cref{alg:Anderson}.}}
\State{\textbf{Set:} $x_{inner\_initial}=x_{k+1/2}^m$}
\For{$j=0,1,2,\cdots, iterN$}
\State{\textbf{Set:} $\mathbf{x_{j+1}\leftarrow}$ \textbf{applying Picard iteration}.}
\EndFor
\State{\textbf{Set:} $x_{k+1}=x_{iterN}$}
\EndFor
\end{algorithmic}
\end{algorithm}
Suppose we just do a single inner loop iteration in \Cref{alg:Anderson_picard_mul}, the total amount of work of $AA(m,AA(0))$ in each iteration is much less than that of applying $sAA(m)$ twice. \Cref{alg:Anderson_picard_mul} also means that we may `turn off' the acceleration for a while and then turn on the acceleration. However, the performance can be better than $sAA(m)$, see our numerical experiments in \Cref{sec:experiments}.

More generally, we apply $sAA(m)$ in the outer loop and apply sAA(n) in the inner loop. So, in each iteration, after applying $sAA(m)$, we get,
$$\displaystyle x_{k+1/2}=\sum_{i=0}^{m_k}\alpha_{i}^{(k)}g(x_{k-m_k+i}).$$
Then we apply $sAA(n)$ with the initial guess $x_0=x_{k+1/2}$ for $iterN$ iterations:
$$x_{k+1}\leftarrow \rm applying\ sAA(n)\ with\ intial\ guess\ x_0=x_{k+1/2}.$$
In other words, the multiplicative composition reads
$$x_{k+1}=sAA(m,sAA(n)).$$
We summarize this in the following algorithm in \Cref{alg:Anderson_AA_mul}.
\begin{algorithm}
\caption{Anderson acceleration with dynamic window-sizes: $\mathbf{AA(m,AA(n))}$}
\label{alg:Anderson_AA_mul}
\begin{algorithmic}
\State{Given: $x_0$, $m$, $n$, $iterM$ and $iterN$.}
\State{Set: $x_{1}=g(x_0).$}
\For{$k=0,1,2,\cdots \ iterM $}
\State{\textbf{Set:} $\mathbf{x_{k+1/2}\leftarrow}$ \textbf{applying} $\mathbf{sAA(m)}$ \textbf{as given in \Cref{alg:Anderson}.}}
\State{\textbf{Set:} $x_{inner\_initial}=x_{k+1/2}$}
\For{$j=0,1,2,\cdots \ iterN$}
\State{\textbf{Set:} $\mathbf{x_{j+1}\leftarrow}$ \textbf{applying} $\mathbf{sAA(n)}$ \textbf{as given in \Cref{alg:Anderson}.}}
\EndFor
\State{\textbf{Set:} $x_{k+1}=x_{iterN}$}
\EndFor
\end{algorithmic}
\end{algorithm}
\begin{remark}There is a lot of variety here. Let $m$ and $n$ be the window size used in the outer loop and inner loop, respectively. And $iterM$ and $iterN$ be the total numbers of iterations for the outer loop and inner loop, respectively. In the present work, we report some results for the case where $m>n$ and $iterM\gg iterN$, which means the window size used in the inner loop is smaller than that used in the outer loop and the maximum iterations of the inner loop is much smaller than that of the outer loop.  For example, one can choose $n=1$ and $iterN=1$. In this case, we could compare the convergence results of $AA(m,AA(1))$ with that of applying $2*iterM$ times of $AA(m)$. Another way to compare them is to calculate the residual per function evaluation. See more discussions in \Cref{sec:experiments}. 
\end{remark}

Moreover, we summarize the memory requirements for those algorithms in \Cref{table_0}. For some problems, memory storage could be crucial. Our numerical results show that the non-stationary AA methods with smaller window sizes usually perform better than the stationary AA algorithm with very large window sizes, which means our proposed non-stationary AA methods may significantly reduce the storage requirements. See more discussions in our numerical experiments in \Cref{sec:experiments}.

\begin{table}[ht]\label{table_0}
\caption{Memory requirements} % title of Table
\centering % used for centering table
\begin{tabular}{c c} % centered columns (4 columns)
\hline\hline %inserts double horizontal lines
Methods& \ \ \ \ \ \ \ memory  \\ [0.1ex] % inserts table
%heading
\hline % inserts single horizontal line
AA(m) &  $m+1$ vectors in memory\\
AA(m)+AA(n)& $\max\{m+1,n+1\}$ vectors in memory\\
AA(m,AA(n)) & $m+n+2$ vectors in memory  \\[0.5ex] % [1ex] adds vertical space
\hline %inserts single line
\end{tabular}
\label{table:t0} % is used to refer this table in the text
\end{table}

\subsection{Non-stationary AA with optimized damping}
\label{subsec:main-2}
In this part, we briefly summarize some of our recent works on Anderson acceleration with optimized damping (AAoptD), which will be used to produce a fully non-stationary Anderson acceleration in the next section. Suppose that $\beta_k$ is different at each iteration $k$, then we have
\begin{eqnarray}\label{eq:1}
x_{k+1}&=&(1-\beta_k)\sum_{i=0}^{m_k}\alpha_{i}^{(k)}x_{k-m_k+i}+\beta_k\sum_{i=0}^{m_k}\alpha_{i}^{(k)}g(x_{k-m_k+i})\nonumber\\
&=&\sum_{i=0}^{m_k}\alpha_{i}^{(k)}x_{k-m_k+i}+\beta_k\left(\sum_{i=0}^{m_k}\alpha_{i}^{(k)}g(x_{k-m_k+i})-\sum_{i=0}^{m_k}\alpha_{i}^{(k)}x_{k-m_k+i}\right).
\end{eqnarray}
Let us define the following averages given by the solution $\alpha^{k}$ to the optimization problem by
\begin{equation}\label{eq:2}
x_{k}^{\alpha}=\sum_{i=0}^{m_k}\alpha_{i}^{(k)}x_{k-m_k+i},\ \ \ \tilde{x}_{k}^{\alpha}=\sum_{i=0}^{m_k}\alpha_{i}^{(k)}g(x_{k-m_k+i}).
\end{equation}
Thus, \eqref{eq:1} becomes
\begin{equation}
x_{k+1}=x_{k}^{\alpha}+\beta_k(\tilde{x}_{k}^{\alpha}-x_{k}^{\alpha}).
\end{equation}
A natural way to choose ``best'' $\beta_k$ at this stage is that choosing $\beta_k$ such that $x_{k+1}$ gives a minimal residual. So we just need to solve the following unconstrained optimization problem:
\begin{equation}\label{eq:3}
\min_{\beta_k}\|x_{k+1}-g(x_{k+1})\|_2=\min_{\beta_k}\|x_{k}^{\alpha}+\beta_k(\tilde{x}_{k}^{\alpha}-x_{k}^{\alpha})-g(x_{k}^{\alpha}+\beta_k(\tilde{x}_{k}^{\alpha}-x_{k}^{\alpha}))\|_2.
\end{equation}
Using the fact that
\begin{eqnarray}
g(x_{k}^{\alpha}+\beta_k(\tilde{x}_{k}^{\alpha}-x_{k}^{\alpha}))&\approx&g(x_{k}^{\alpha})+\beta_k\frac{\partial g}{\partial x}\Big|_{x_k^{\alpha}}(\tilde{x}_{k}^{\alpha}-x_{k}^{\alpha})\nonumber\\
&\approx&g(x_{k}^{\alpha})+\beta_k\left(g(\tilde{x}_{k}^{\alpha})-g(x_{k}^{\alpha})\right).
\end{eqnarray}
Therefore, \eqref{eq:3} becomes
\begin{eqnarray}\label{eq:4}
\min_{\beta_k}\|x_{k+1}-g(x_{k+1})\|_2&=&\min_{\beta_k}\|x_{k}^{\alpha}+\beta_k(\tilde{x}_{k}^{\alpha}-x_{k}^{\alpha})-g(x_{k}^{\alpha}+\beta_k(\tilde{x}_{k}^{\alpha}-x_{k}^{\alpha}))\|_2\nonumber\\
&\approx&\min_{\beta_k}\|x_{k}^{\alpha}+\beta_k(\tilde{x}_{k}^{\alpha}-x_{k}^{\alpha})-\left[g(x_{k}^{\alpha})+\beta_k(g(\tilde{x}_{k}^{\alpha})-g(x_{k}^{\alpha}))\right]\|_2\nonumber\\
&\approx&\min_{\beta_k}\|\left(x_{k}^{\alpha}-g(x_{k}^{\alpha})\right)-\beta_k\left[(g(\tilde{x}_{k}^{\alpha})-g(x_{k}^{\alpha}))-(\tilde{x}_{k}^{\alpha}-x_{k}^{\alpha})\right]\|_2.
\end{eqnarray}
Thus, we just need to calculate the projection
\begin{equation}\label{eq:55}
\beta_k=\Big|\frac{\left(x_{k}^{\alpha}-g(x_{k}^{\alpha})\right)\cdot\left[\left(x_{k}^{\alpha}-g(x_{k}^{\alpha})\right)-(\tilde{x}_{k}^{\alpha}-g(\tilde{x}_{k}^{\alpha}))\right]}{\|\left[\left(x_{k}^{\alpha}-g(x_{k}^{\alpha})\right)-(\tilde{x}_{k}^{\alpha}-g(\tilde{x}_{k}^{\alpha}))\right]\|_2}\Big|.
\end{equation}

Anderson acceleration algorithm with optimized damping $AAoptD(m)$ reads as follows:
\begin{algorithm}
\caption{Anderson acceleration with optimized dampings: $\mathbf{AAoptD(m)}$}
\label{alg:AAoptD}
\begin{algorithmic}
\State{Given: $x_0$ and $m\geq 1$.}
\State{Set: $x_{1}=g(x_0).$}
\For{$k=0,1,2,\cdots $}
\State{Set: $m_{k}=\min\{m,k\}$.}
\State{Set: $F_k=(f_{k-m_k},\cdots,f_k)$, where $f_i=g(x_i)-x_i$.}
  \State{Determine: $\alpha^{(k)}=\left(\alpha_0^{(k},\cdots,\alpha_{m_k}^{(k)}\right)^{T}$ that solves }
\State{$\ \ \ \ \ \ \ \ \ \ \ \displaystyle \min_{\alpha=(\alpha_0,\cdots,\alpha_{m_k})^{T}}\|F_k\alpha\|_2$ $\ s.\ t.$ $\displaystyle\sum_{i=0}^{m_k}\alpha_i=1.$}
\State{Set: $\ \ \displaystyle x_{k}^{\alpha}=\sum_{i=0}^{m_k}\alpha_{i}^{(k)}x_{k-m_k+i},\ \ \ \tilde{x}_{k}^{\alpha}=\sum_{i=0}^{m_k}\alpha_{i}^{(k)}g(x_{k-m_k+i}).$}
\State{Set: $\ \ \displaystyle {r_p}=\left(x_{k}^{\alpha}-g(x_{k}^{\alpha})\right),\ \ {r_q}=\left(\tilde{x}_{k}^{\alpha}-g(\tilde{x}_{k}^{\alpha})\right)$.}
\State{Set: $\ \ \ \ \ \ \ \ \ \ \ \displaystyle \beta_k=\frac{(r_p-r_q)^{T}r_p}{\|r_p-r_q\|_2}$.}
\State{Set: $ \ \ \displaystyle x_{k+1}=x_{k}^{\alpha}+\beta_k(\tilde{x}_{k}^{\alpha}-x_{k}^{\alpha})$.}
\EndFor
\end{algorithmic}
\end{algorithm}

\begin{remark}
This optimized damping step is a ``local optimal'' strategy at $k$th iteration. It usually will speed up the convergence rate compared with an undamped one, but not always. Because in $(k+1)$th iteration, it uses a combination of all previous m history information. Moreover, when $\beta_k$ is very close to zero, the system is overdamped, which sometimes may also slow down the convergence speed. We may need to further modify our $\beta_k$ to bound it away from zero by using
\begin{equation}\label{damp2}
\hat{\beta}_k =\max\{\beta_k,\eta\},
\end{equation}
or \begin{equation}\label{damp1}
\hat{\beta}_k =\begin{cases}
\beta_k &\text{if $\beta_k \geq \eta$},\\
1-\beta_k &\text{if $\beta_k< \eta$}.\\
\end{cases}
\end{equation}
where $\eta$ is a small positive number such that $0<\eta<0.5$. For more details on the implementation of $AAoptD(m)$ and its performance, we refer the readers to our recent paper.\cite{Chen2022} 
\end{remark}

\subsection{Fully Non-stationary AA with dynamic window-sizes and optimized damping}
\label{subsec:main-3} At this stag{}e, we are ready to present our final fully non-stationary Anderson acceleration algorithms. Since alternating the window sizes and using different damping factors in each iteration are two main ways to produce non-stationary Anderson acceleration, we need to combine those strategies to make it into a fully non-stationary algorithm. In this part, the additive composition and the multiplicative composition are used to combine $AA(m)$ and $AAoptD(n)$. For example, we can obtain the following fully non-stationary algorithm $AA(m)+AAoptD(n)$ as in \Cref{alg:Anderson_AAoptD_sum} and $AA(m,AAoptD(n))$ as in \Cref{alg:Anderson_AAoptD_mul} by using additive and multiplicative composite combination, respectively.
\begin{algorithm}
\caption{Fully non-stationary Anderson acceleration: $\mathbf{AA(m)+AAoptD(n)}$}
\label{alg:Anderson_AAoptD_sum}
\begin{algorithmic}
\State{Given: $x_0$, $\alpha_m$, $\alpha_n$ and $m,n$.}
\State{Set: $x_{1}=g(x_0).$}
\For{$k=0,1,2,\cdots, $}
\State{\textbf{Set:} $\mathbf{x_{k+1/2}^m\leftarrow}$ \textbf{applying} $\mathbf{sAA(m)}$ \textbf{as given in \Cref{alg:Anderson}.}}
\State{\textbf{Set:} $\mathbf{x_{k+1/2}^n\leftarrow}$ \textbf{applying} $\mathbf{AAoptD(n)}$ \textbf{as given in \Cref{alg:AAoptD}.}}
\State{\textbf{Set:} $\mathbf{x_{k+1}=\alpha_m x_{k+1/2}^m+\alpha_n x_{k+1/2}^n}$.}
\EndFor
\end{algorithmic}
\end{algorithm}
\begin{algorithm}
\caption{Fully non-stationary Anderson acceleration: $\mathbf{AA(m,AAoptD(n))}$}
\label{alg:Anderson_AAoptD_mul}
\begin{algorithmic}
\State{Given: $x_0$, $m$ $n$ $iterM$ and $iterN$.}
\State{Set: $x_{1}=g(x_0).$}
\For{$k=0,1,2,\cdots  \ iterM$}
\State{\textbf{Set:} $\mathbf{x_{k+1/2}\leftarrow}$ \textbf{applying} $\mathbf{sAA(m)}$ \textbf{as given in \Cref{alg:Anderson}.}}
\State{\textbf{Set:} $x_{inner\_initial}=x_{k+1/2}$}
\For{$j=0,1,2,\cdots \ iterN$}
\State{\textbf{Set:} $\mathbf{x_{j+1}\leftarrow}$ \textbf{applying} $\mathbf{AAoptD(n)}$ \textbf{as given in \Cref{alg:AAoptD}.}}
\EndFor
\State{\textbf{Set:} $x_{k+1}=x_{iterN}$}
\EndFor
\end{algorithmic}
\end{algorithm}

Similarly, $AAoptD(m+n)$, $AAoptD(m,AAoptD(n))$ or 
$AAoptD(m,AA(n))$ can be easily obtained. Now, we have a set of new non-stationary AA algorithms. We test and compare the performance of some of these methods with stationary AA in \Cref{sec:experiments}. Suggestions on how to use and choose these methods are provided in our final conclusion part in \Cref{sec:conclusions}.

\section{Convergence rate}\label{sec:theory}
In this section, we investigate the rate of convergence of these non-stationary Anderson acceleration methods. The main technical results and assumptions are adopted from papers\cite{ToKe2015,Ev2020} with necessary modifications. Here we provide the convergence theorems for $AAoptD(m)$, $AA(m,AA(1))$ with inner loop iteration $iterN=1$ and $AAoptD(m,AA(1))$ with inner loop iteration $iterN=1$. These typical non-stationary AA methods are extensively studied in \Cref{sec:experiments}. Similarly, one can derive the rate of convergence to other non-stationary AA methods.

To begin with, we summarize the convergence result in \Cref{thm:bigthm} for Anderson acceleration with optimized damping as in \Cref{alg:AAoptD}. 
\begin{theorem}\label{thm:bigthm}Assume that
\begin{itemize}
 \item $g: R^n\rightarrow R^n$ has a fixed point $x^*\in R^n$ such that $g(x^*)=x^*$. 
 \item $g$ is uniformly Lipschitz continuously differentiable in the ball $B(\rho)=\{u|\|x-x^*\|_2\leq\rho\}$.
 \item There exists $\kappa\in(0,1)$ such that $\|g(y)-g(x)\|_2\leq\kappa\|y-x\|_2$ for all $x, y\in R^n$. 
 \item Suppose that $\exists M$ and $\epsilon>0$ such that for all $k>m$, $\sum_{i=0}^{m-1}|\alpha_i|<M$ and $|\alpha_m|\geq\epsilon$. 
\end{itemize}
 Then we have the following convergence rate for $AAoptD(m)$ given in \Cref{alg:AAoptD}:
\begin{equation}
\|f(x_{k+1})\|_2\leq\theta_{k+1}\left[(1-\beta_k)+\kappa\beta_k\right]\|f(x_k)\|_2+\sum_{i=0}^{m} O(\|f(x_{k-m+i})\|_2^2),
\end{equation}
where the average gain
$$\theta_{k+1}=\frac{\|\sum_{i=0}^{m}\alpha_i^kf(x_{k-m+i})\|_2}{\|f(x_k)\|_2}$$
and
$$
\beta_k=\Big|\frac{\left(x_{k}^{\alpha}-g(x_{k}^{\alpha})\right)\cdot\left[\left(x_{k}^{\alpha}-g(x_{k}^{\alpha})\right)-(\tilde{x}_{k}^{\alpha}-g(\tilde{x}_{k}^{\alpha}))\right]}{\|\left[\left(x_{k}^{\alpha}-g(x_{k}^{\alpha})\right)-(\tilde{x}_{k}^{\alpha}-g(\tilde{x}_{k}^{\alpha}))\right]\|_2}\Big|
$$
with 
$$x_{k}^{\alpha}=\sum_{i=0}^{m_k}\alpha_{i}^{(k)}x_{k-m_k+i},\ \ \ \tilde{x}_{k}^{\alpha}=\sum_{i=0}^{m_k}\alpha_{i}^{(k)}g(x_{k-m_k+i}).$$
\end{theorem}
\begin{proof} The proof of this theorem can be found in this paper\cite{Ev2020} for general damping $\beta_k$. The key ideas of the analysis are relating the difference of consecutive iterates to residuals based on performing the inner-optimization and explicitly defining the gain in the optimization stage to be the ratio of improvement over a step of the unaccelerated fixed-point iteration. Additionally, here we also need to use \eqref{eq:1} and \eqref{eq:55} to explicitly calculate these optimized $\beta_k$.
\end{proof}

Next, we provide the convergence rate for non-stationary Anderson acceleration methods $AA(m,AA(1))$ with inner loop iteration $iterN=1$ and $AAoptD(m,AA(1))$ with inner loop iteration $iterN=1$, which are extensively studied in \Cref{sec:experiments}.

\begin{theorem} \label{thm:thm1} Assume that $g: R^n\rightarrow R^n$ has a fixed point $x^*\in R^n$ such that $g(x^*)=x^*$ and satisfies all assumptions in \Cref{thm:bigthm}. Then we have the following convergence rate for $AA(m,AA(1))$ as in \Cref{alg:Anderson_AA_mul} with inner loop iteration $iterN=1$:
\begin{equation}
\|f(x_{k+1})\|_2\leq\bar{\theta}_{1}\theta_{k+1}\kappa^2\|f(x_k)\|_2+high\ order\ terms,
%\sum_{i=0}^{m} O(\|f(x_{k-m+i})\|_2^2)
\end{equation}
where
$$\theta_{k+1}=\frac{\|\sum_{i=0}^{m}\alpha_i^kf(x_{k-m+i})\|_2}{\|f(x_k)\|_2},\ \ \ \bar{\theta}_{1}=\frac{\|\bar{\alpha}_0^1f(\bar{x}_{0})+\bar{\alpha}_1^1f(\bar{x}_{1})\|_2}{\|f(\bar{x}_1)\|_2}$$
with
 $$ \bar{x}_0=x_{k+1/2}, \ \ \ \bar{x}_1=g(\bar{x}_0).$$
\end{theorem}

\begin{proof}
For the outer loop, according to the results in \Cref{thm:bigthm} without damping (i.e. $\beta_k=1$), we have
\begin{equation}\label{kw_0}
\|f(x_{k+1/2})\|_2\leq\theta_{k+1}\kappa\|f(x_k)\|_2+\sum_{i=0}^{m} O(\|f(x_{k-m+i})\|_2^2),
\end{equation}
where
$$\theta_{k+1}=\frac{\|\sum_{i=0}^{m}\alpha_i^kf(x_{k-m+i})\|_2}{\|f(x_k)\|_2}.$$
As $\alpha_i^k$ is the solution to the optimization problem in \Cref{alg:Anderson} and the fact that $\alpha_k^k=1$, $\alpha_j^k=0, j\not=k$, is in the feasible set for the optimization problem, we immediately have
$$0\leq\theta_{k+1}\leq 1.$$
For the inner loop with $iterN=1$, we have the initial guess $\bar{x}_0=x_{k+1/2}$, then $\bar{x}_1=g(\bar{x}_0)$ and
$$f(\bar{x}_0)=g(\bar{x}_0)-\bar{x}_0,\ \ \ f(\bar{x}_1)=g(\bar{x}_1)-\bar{x}_1.$$
Let $\bar{\alpha}_0$ and $\bar{\alpha}_1$ be the solution to the inner loop optimization problem, then applying \Cref{thm:bigthm} with $m=1$ without damping (i.e. $\beta_k=1$), we have
\begin{equation}\label{kw_1}
\|f(\bar{x}_2)\|_2\leq\bar{\theta}_1\kappa\|f(\bar{x}_1)\|+O(\|f(\bar{x}_{1})\|_2^2)+O(\|f(\bar{x}_{0})\|_2^2)
\end{equation}
with
$$ \bar{\theta}_{1}=\frac{\|\bar{\alpha}_0^1f(\bar{x}_{0})+\bar{\alpha}_1^1f(\bar{x}_{1})\|_2}{\|f(\bar{x}_1)\|_2},$$
where $\bar{\alpha}_0^1$ and $\bar{\alpha}_1^1$ solves the optimization problem of $AA(1)$ in the inner loop iteration. Similarly, since $\bar{\alpha}_0^1=0$ and $\bar{\alpha}_1^1=1$, is in the feasible set for the related optimization problem, we get
$$0\leq\bar{\theta}_{1}\leq 1.$$
Using \eqref{kw_0} and the fact that the inner loop use $\bar{x}_0=x_{k+1/2}$ as an initial guess,  we have $\bar{x}_1=g(\bar{x}_0)=g(x_{k+1/2})$. Therefore, 
\begin{equation}\label{kw_3}
\|f(\bar{x}_1)\|_2\leq\theta_{k+1}\kappa\|f(x_k)\|_2+\sum_{i=0}^{m} O(\|f(x_{k-m+i})\|_2^2).
\end{equation}
Since $iterN=1$, so after finishing the inner loop iteration, we will set $x_{k+1}=\bar{x}_2$. Thus, from \eqref{kw_1} and \eqref{kw_3}, we finally obtain
\begin{equation}
\|f(x_{k+1})\|_2=\|f(\bar{x}_2)\|_2\leq\bar{\theta}_{1}\theta_{k+1}\kappa^2\|f(x_k)\|_2+high\ order\ terms,
%\sum_{i=0}^{m} O(\|f(x_{k-m+i})\|_2^2)
\end{equation}
which completes the proof of this theorem.
\end{proof}

\begin{theorem} \label{thm:thm2} Assume that $g: R^n\rightarrow R^n$ has a fixed point $x^*\in R^n$ such that $g(x^*)=x^*$ and satisfies the assumptions in \Cref{thm:bigthm}. Then we have the following convergence rate for $AAoptD(m,AA(1))$ with inner loop iteration $iterN=1$:
\begin{equation}
\|f(x_{k+1})\|_2\leq\bar{\theta}_{1}\theta_{k+1}\kappa\left[(1-\beta_k)+\kappa\beta_k\right]\|f(x_k)\|_2+high\ order\ terms,
\end{equation}
where
$$
\beta_k=\Big|\frac{\left(x_{k}^{\alpha}-g(x_{k}^{\alpha})\right)\cdot\left[\left(x_{k}^{\alpha}-g(x_{k}^{\alpha})\right)-(\tilde{x}_{k}^{\alpha}-g(\tilde{x}_{k}^{\alpha}))\right]}{\|\left[\left(x_{k}^{\alpha}-g(x_{k}^{\alpha})\right)-(\tilde{x}_{k}^{\alpha}-g(\tilde{x}_{k}^{\alpha}))\right]\|_2}\Big|.
$$
and 
$$\theta_{k+1}=\frac{\|\sum_{i=0}^{m}\alpha_if(x_{k-m+i})\|_2}{\|f(x_k)\|_2},\ \ \ \bar{\theta}_{1}=\frac{\|\bar{\alpha}_0f(\bar{x}_{0})+\bar{\alpha}_1f(\bar{x}_{1})\|_2}{\|f(\bar{x}_1)\|_2}$$
with
 $$ \bar{x}_0=x_{k+1/2}, \ \ \ \bar{x}_1=g(\bar{x}_0).$$
\end{theorem}
\begin{proof}The proof is similar to the proof of \Cref{thm:thm1}, thus we omit it here.
\end{proof}

\section{Experimental results and discussion}
\label{sec:experiments} In this section, we numerically compare the performance of these fully non-stationary Anderson acceleration algorithms with regular stationary AA (with uniform window size and a constant damping factor). All these experiments are done in the MATLAB 2021b environment. MATLAB codes are available upon request to the authors.

This first example is from Walker and Ni's paper,\cite{WaNi2011}  where a stationary Anderson acceleration with window size $m=50$ is used to solve the Bratu problem. This problem has a long history, we refer the reader to Glowinski et al.\cite{Gl1985} and Pernice and Walker,\cite{Pe1998} and the references in those papers.
\begin{example}\textbf{The\ Bratu\ problem.} The Bratu problem is a nonlinear PDE boundary value problem as follows:
\begin{eqnarray*}
\Delta u +\lambda\ e^{u}&=&0,\ \ in \ \ D=[0,1]\times[0,1],\nonumber\\
u&=&0,\ \ on \ \ \partial D. 
\end{eqnarray*}
\end{example}
In this experiment, we used a centered-difference discretization on a $64\times 64$ grid. We take $\lambda=6$ in the Bratu problem and use the zero initial approximate solution in all cases. And we also applied the preconditioning such that the basic Picard iteration still works. The preconditioning matrix that we used here is the diagonal inverse of the matrix $A$, where $A$ is a matrix for the discrete Laplace operator.

We first solve the Bratu problem using the additive composition of non-stationary AA methods. In order to compare the results with stationary AA, we consider the residual achieved in one iteration of fully non-stationary AA methods as it is achieved in ``two iterations'', since additive composition strategy combines two regular AA in one iteration. We already doubled the iterations for non-stationary AA methods in the following plots. 
Here we choose our window sizes $m=20$ and $n=1$, thus the total amount of work is less than that of applying the staionary $AA(20)$ twice. Secondly, we solve the same problem by using the multiplicative composition of non-stationary AA methods. The window sizes for the outer loop and inner loop are $m=20$ and $n=1$, respectively. Here, we apply only one iteration in the inner loop. So we can easily compare their convergence rates with regular AA just like the way used in the case of the additive composition. Moreover, in some cases, function evaluations are more expensive. Since non-stationary AA methods usually involve more function evaluations per iteration, we also plot the residual per function evaluation, which can also help us directly compare the performance of non-stationary AA and stationary AA. The results are shown in \Cref{fig:fig_3}, \Cref{fig:fig_4} and \Cref{fig:fig_new_4}. The total time used for different methods is shown in \Cref{table:t2}. Here we use the Matlab commands $tic$ and $toc$ to record the running time for each of these methods and we have not optimized our Matlab codes yet. 

From \Cref{fig:fig_3}, \Cref{fig:fig_4}, \Cref{fig:fig_new_4} and \Cref{table:t2}, non-stationary Anderson acceleration methods work better than regular stationary AA. The best one is the fully non-stationary Anderson acceleration $AAoptD(20,AA(1))$, which means the window size for the outer loop is $20$ and the window size for the inner loop is $1$. To compare the results in \cite{WaNi2011}, we also solve this problem with larger window size $m$, the convergence result is in \Cref{fig:fig_new_k6}. From \Cref{fig:fig_new_k6}, we find that the non-stationary AA methods with outer window size $m=20$ and inner loop window size $n=1$ are better than stationary AA with window size $m=20$. Moreover, some of these methods (i.e. $AA(20,AA(1))$ and $AAoptD(20,AA(1))$) are comparable with the stationary $AA(50)$. This is very important for solving larger size problems where computer storage is crucial. Our proposed non-stationary AA methods may save lots of memory storage while maintaining a similar (or even faster) convergence rate. 
\begin{figure}[htbp]
  \centering
  \includegraphics[height=3.8in, width=4.5in]{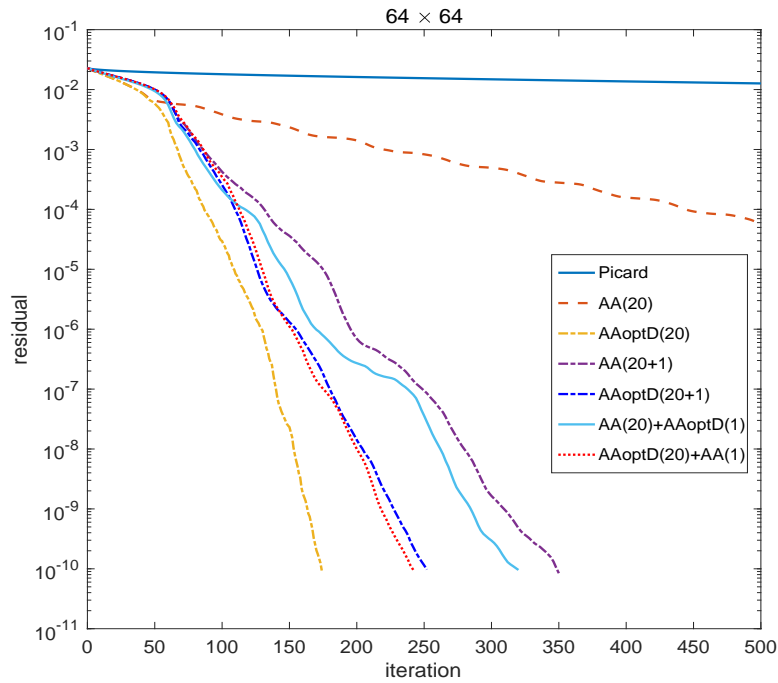}
  \caption{Non-stationary AA methods solve the Bratu problem: additive composition.}
  \label{fig:fig_3}
\end{figure}
\begin{figure}[htbp]
  \centering
  \includegraphics[height=3.8in, width=4.5in]{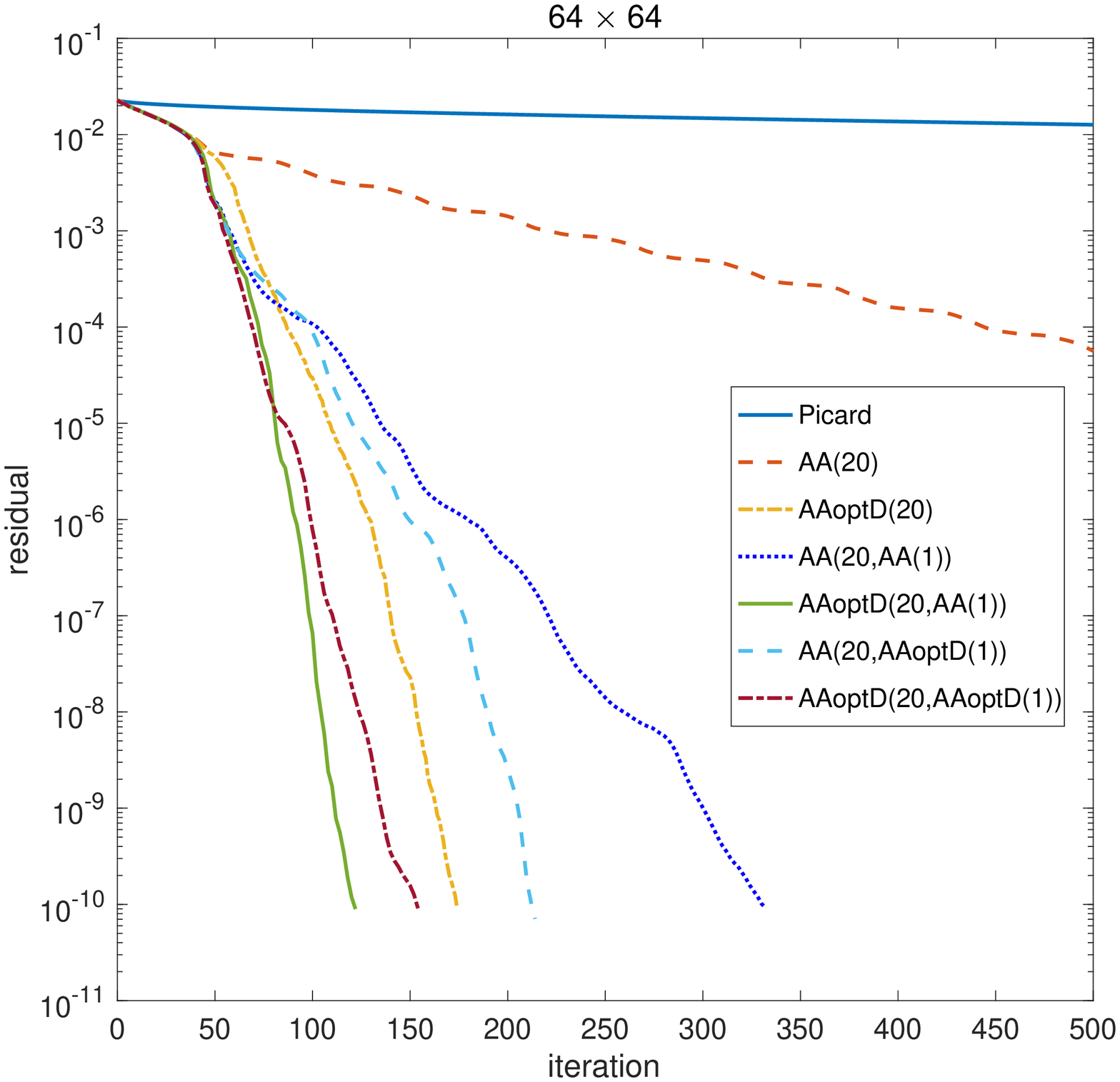}
  \caption{Non-stationary AA methods solve the Bratu problem: multiplicative composition.}
  \label{fig:fig_4}
\end{figure}
\begin{figure}[htbp]
  \centering
  \label{fig:f_new_4}\includegraphics[height=3.8in, width=4.5in]{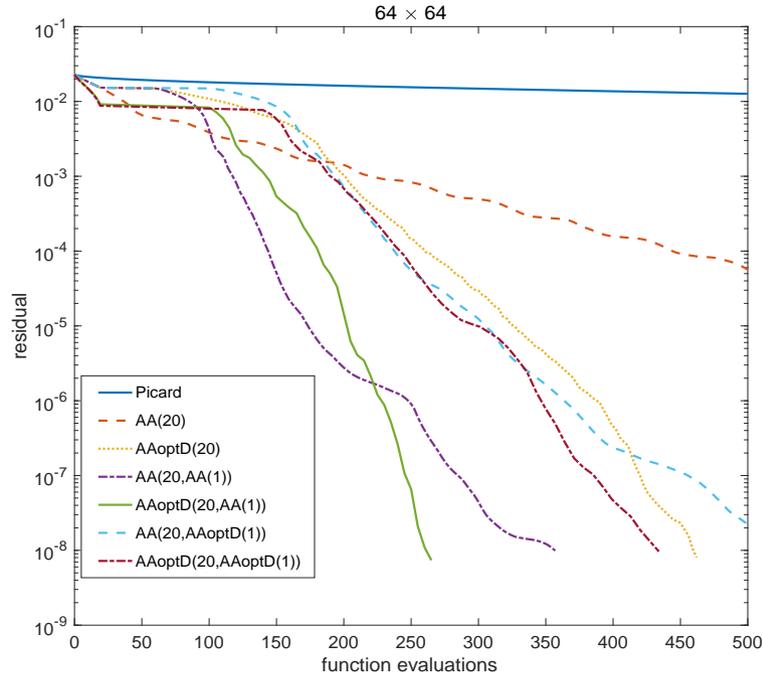}
  \caption{Non-stationary AA methods solve the Bratu problem: residual vs function evaluations (multiplicative composition).}
  \label{fig:fig_new_4}
\end{figure}
\begin{figure}[htbp]
  \centering
  \includegraphics[height=3.8in, width=4.5in]{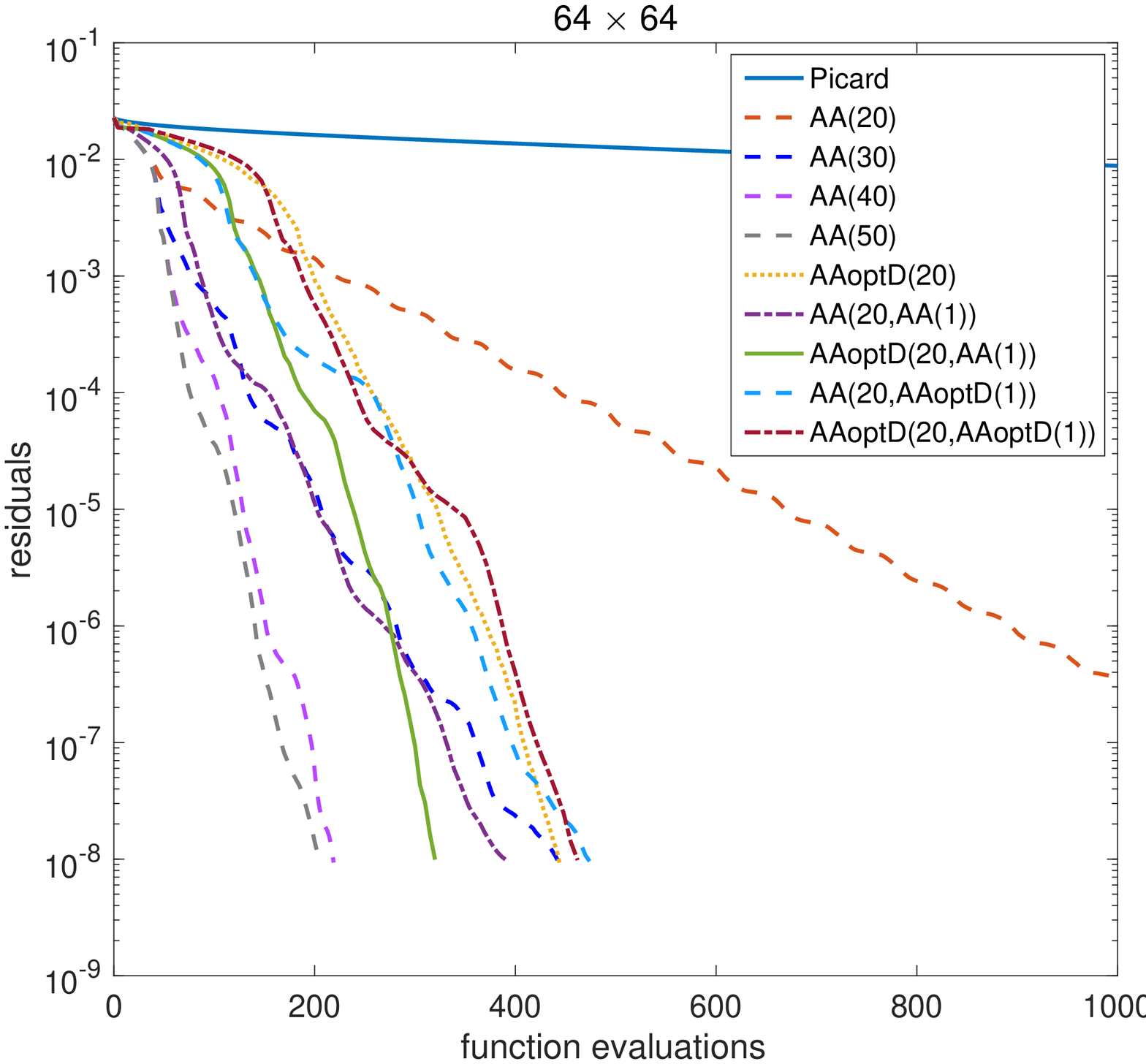}
  \caption{Non-stationary AA methods solve the Bratu problem (multiplicative composition).}
  \label{fig:fig_new_k6}
\end{figure}
% \begin{figure}[htbp]
%   \centering
%   \includegraphics[height=3.6in, width=4.5in]{paper_bratu_64_mul_50.eps}
%   \caption{Non-stationary AA methods solve the Bratu problem: $m=50$ and $n=1$.}
%   \label{fig:fig_new_k5}
% \end{figure}

\begin{table}[ht]
\caption{Top 4: total time used in solving the Bratu problem on a $64 \times 64$ grid.} % title of Table
\centering % used for centering table
\begin{tabular}{c c c c} % centered columns (4 columns)
\hline\hline %inserts double horizontal lines
Methods& times (s) &\ \ Methods& times (s) \\ [0.1ex] % inserts table
%heading
\hline % inserts single horizontal line
AAoptD(20,AA(1)) & 65.46 &AAoptD(20) & 113.39 \\
AA(20,AA(1)) & 107.94 & AA(20+1) & 116.27 \\
AAoptD(20) & 112.86  & AAoptD(20)+AA(1) & 130.46 \\
AAoptD(20,AAoptD(1)) & 114.70 & AA(20)+AAoptD(1) & 173.03 \\ [0.5ex] % [1ex] adds vertical space
\hline %inserts single line
\end{tabular}
\label{table:t2} % is used to refer this table in the text
\end{table}
%%%%%non linear Poisson%%%%
\begin{example}\textbf{The\ stationary\ nonlinear\ convection-diffusion\ problem.} Solve the following 2D nonlinear convection-diffusion equation in a square region:
$$\epsilon(-u_{xx}-u_{yy})+(u_x+u_y)+ku^2=f(x,y),\ \  (x,y)\in D=[0,1]\times[0,1] $$
with the source term 
$$f(x,y)=2\pi^2\sin(\pi x)\sin(\pi y)$$
and zero boundary conditions: $u(x,y)=0$ on $\partial D$.
\end{example}

In this numerical experiment, we use a centered-difference discretization on a $32\times 32$ grid. We choose $\epsilon=1$, $\epsilon=0.1$ and $\epsilon=0.01$, respectively. Those $\epsilon$ values indicate the competition between the diffusion and convection effect. We take $k=3$ in the above problem and use $u_0=(1,1,\cdots,1)^{T}$ as an initial approximate solution in all cases. As in solving the Bratu problem, the same preconditioning strategy is used here. 

For $\epsilon=1$, from \Cref{fig:fig_5}, \Cref{fig:fig_6}, \Cref{fig:fig_new60} and \Cref{table:t3}, we observe similar results that the non-stationary methods $AAoptD(5,AA(1))$ and $AA(5,AAoptD(1))$ perform better than stationary $AA(5)$. For $\epsilon=0.1$, from \Cref{fig:fig_new61}, we find that the Picard iteration does not converge anymore. However, Anderson acceleration methods with a small window size already work. Moreover, some of our proposed non-stationary AA methods (for example, $AAoptD(1,AA(1))$ $AA(1,AAoptD(1))$ and $AA(1,AA(1))$) perform better than the stationary $AA(1)$ method. For $\epsilon=0.01$, from \Cref{fig:fig_new62} and \Cref{fig:fig_new63}, we see that the Picard method, the stationary $AA(1)$ method and non-stationary method $AA(1,AA(1))$ method do not converge while other non-stationary methods still converge. However, for the converging methods, there are some wiggles in the numerical approximate, see the bottom figures in \Cref{fig:fig_new64}. This may result from using the central difference scheme for the convection term. To solve this problem, we then use the the upwind scheme (backward difference) for the convection term. The results are shown in \Cref{fig:fig_new621} and \Cref{fig:fig_new641}. All acceleration methods converge. $AA(1,AA(1))$ performs best. $AAoptD(1,AA(1))$ and $AA(1,AAoptD(1))$ are comparable to the stationary $AA(1)$ method. Moreover, there are no wiggles in the numerical approximate when applying upwind scheme for the convection term.

\begin{figure}[htbp]
  \centering
  \includegraphics[height=3.8in, width=4.5in]{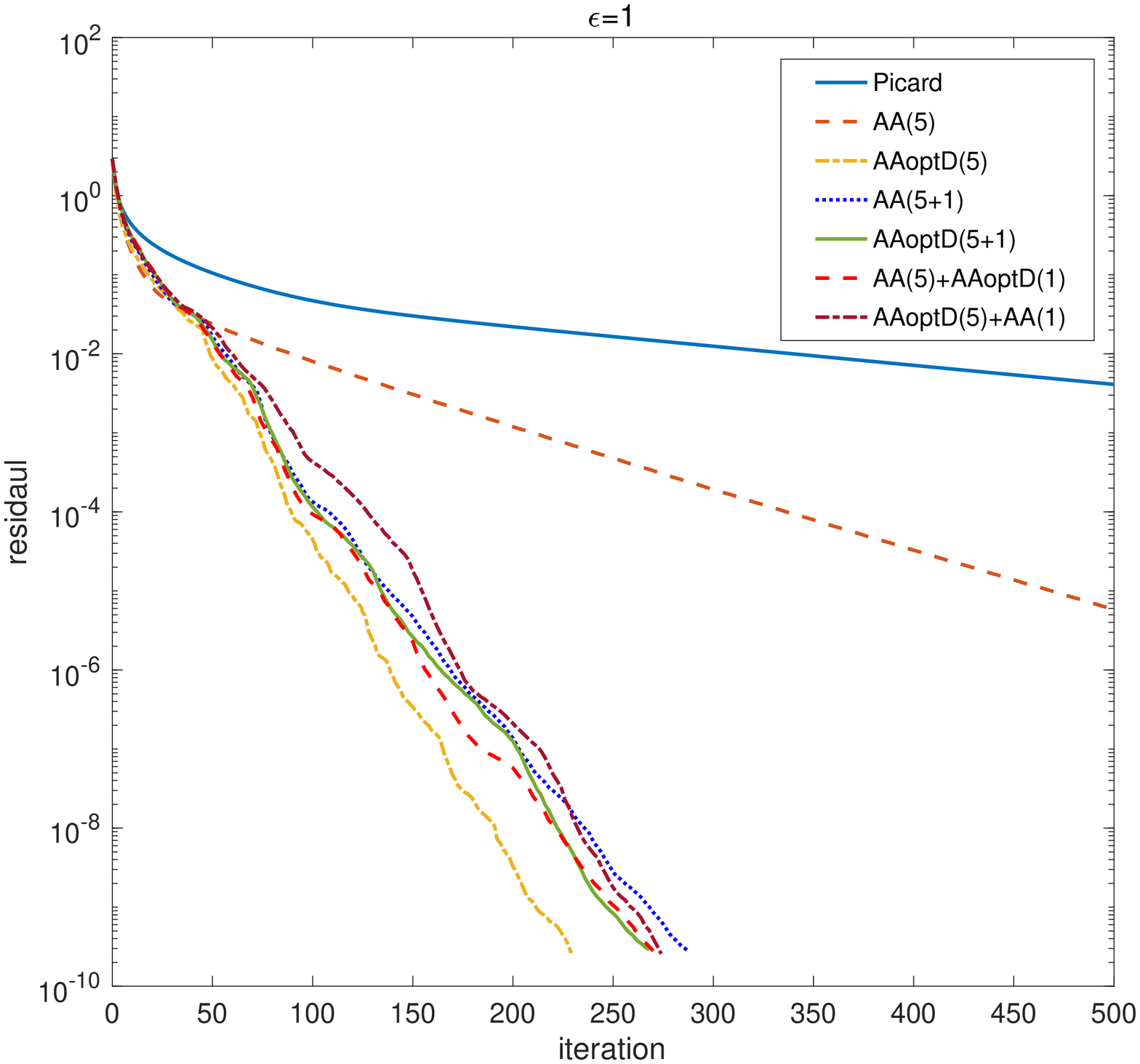}
  \caption{Solve the convection-diffusion problem: additive composition with $\epsilon=1$.}
  \label{fig:fig_5}
\end{figure}

\begin{figure}[htbp]
  \centering
  \includegraphics[height=3.8in, width=4.5in]{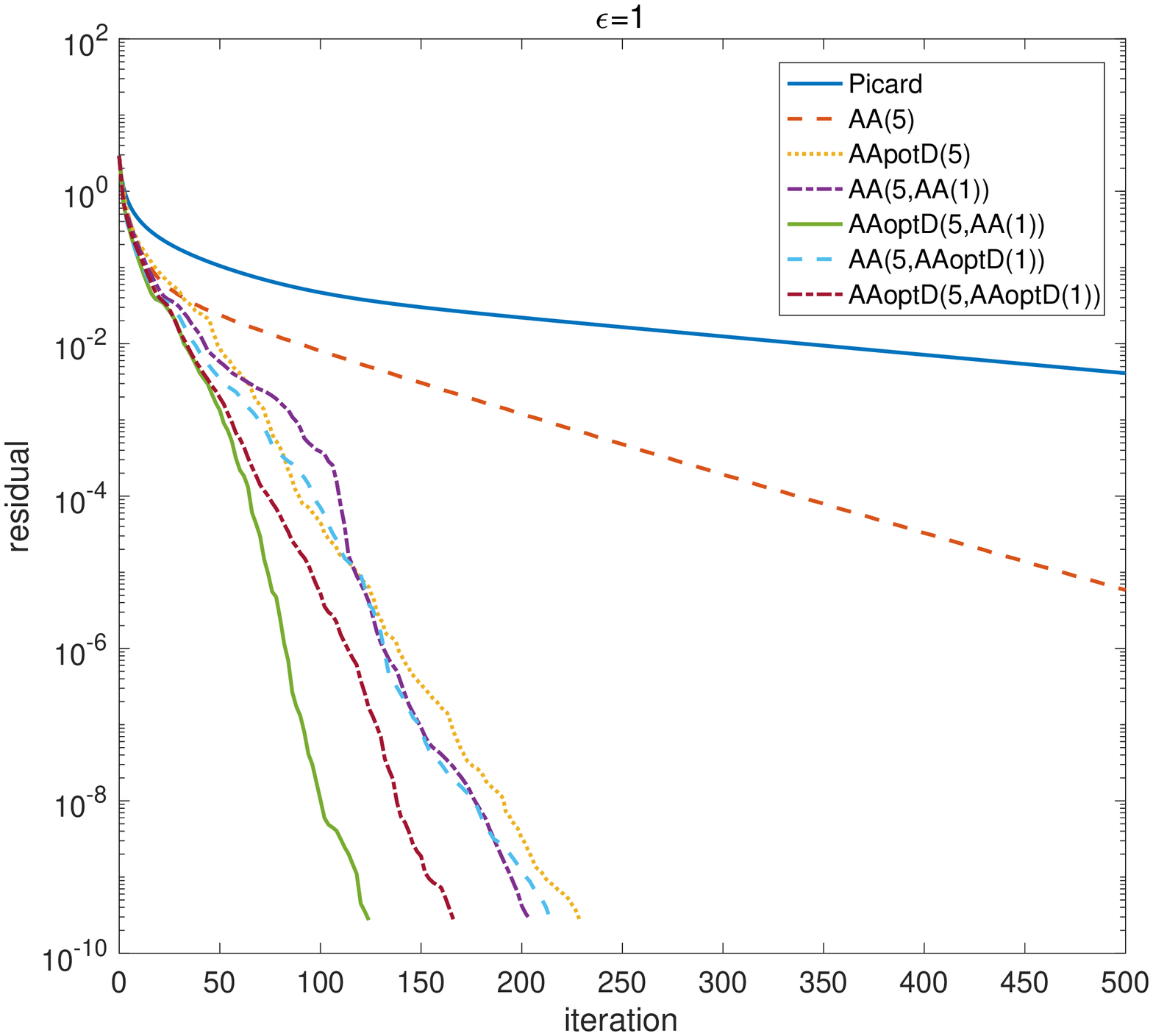}
  \caption{Solve the convection-diffusion problem: multiplicative composition with $\epsilon=1$.}
  \label{fig:fig_6}
\end{figure}

\begin{figure}[htbp]
  \centering
  \includegraphics[height=3.8in, width=4.5in]{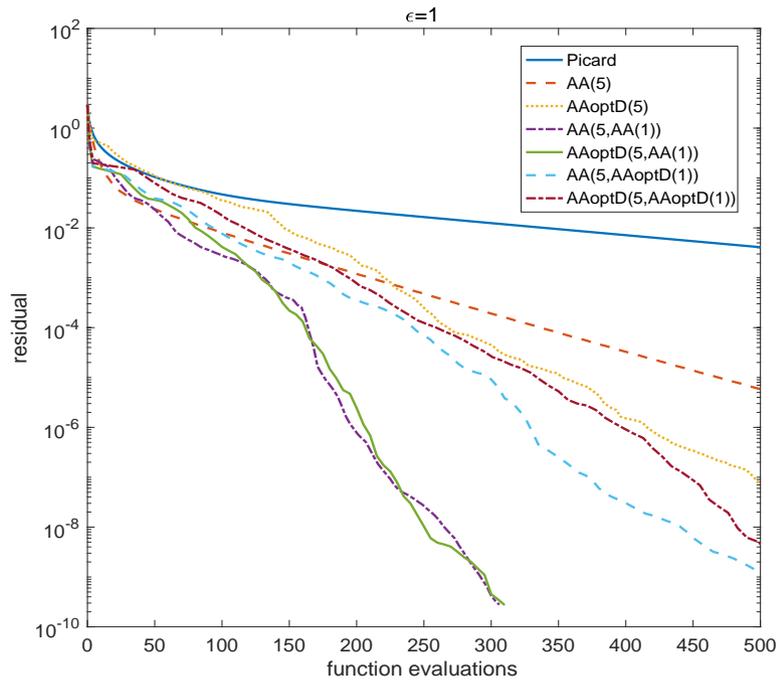}
  \caption{Solve the convection-diffusion problem with $\epsilon=1$: function evaluations.}
  \label{fig:fig_new60}
\end{figure}

\begin{table}[ht]
\caption{Top 4: total time used in solving the convection-diffusion problem with $\epsilon=1$.} % title of Table
\centering % used for centering table
\begin{tabular}{c c c c} % centered columns (4 columns)
\hline\hline %inserts double horizontal lines
Methods& times (s) &\ \ Methods& times (s) \\ [0.1ex] % inserts table
%heading
\hline % inserts single horizontal line
AAoptD(5,AA(1)) & 2.26 &AA(5+1) & 3.29 \\
AA(5,AA(1)) & 2.33 & AAoptD(5) & 4.80 \\
AA(5,AAoptD(1)) & 3.92  & AAoptD(5)+AA(1) & 4.93 \\
AAoptD(5,AAoptD(1)) & 4.14 & AA(5)+AAoptD(1) & 5.00 \\ [0.5ex] % [1ex] adds vertical space
\hline %inserts single line
\end{tabular}
\label{table:t3} % is used to refer this table in the text
\end{table}

\begin{figure}[htbp]
  \centering
  \includegraphics[height=3.8in, width=4.5in]{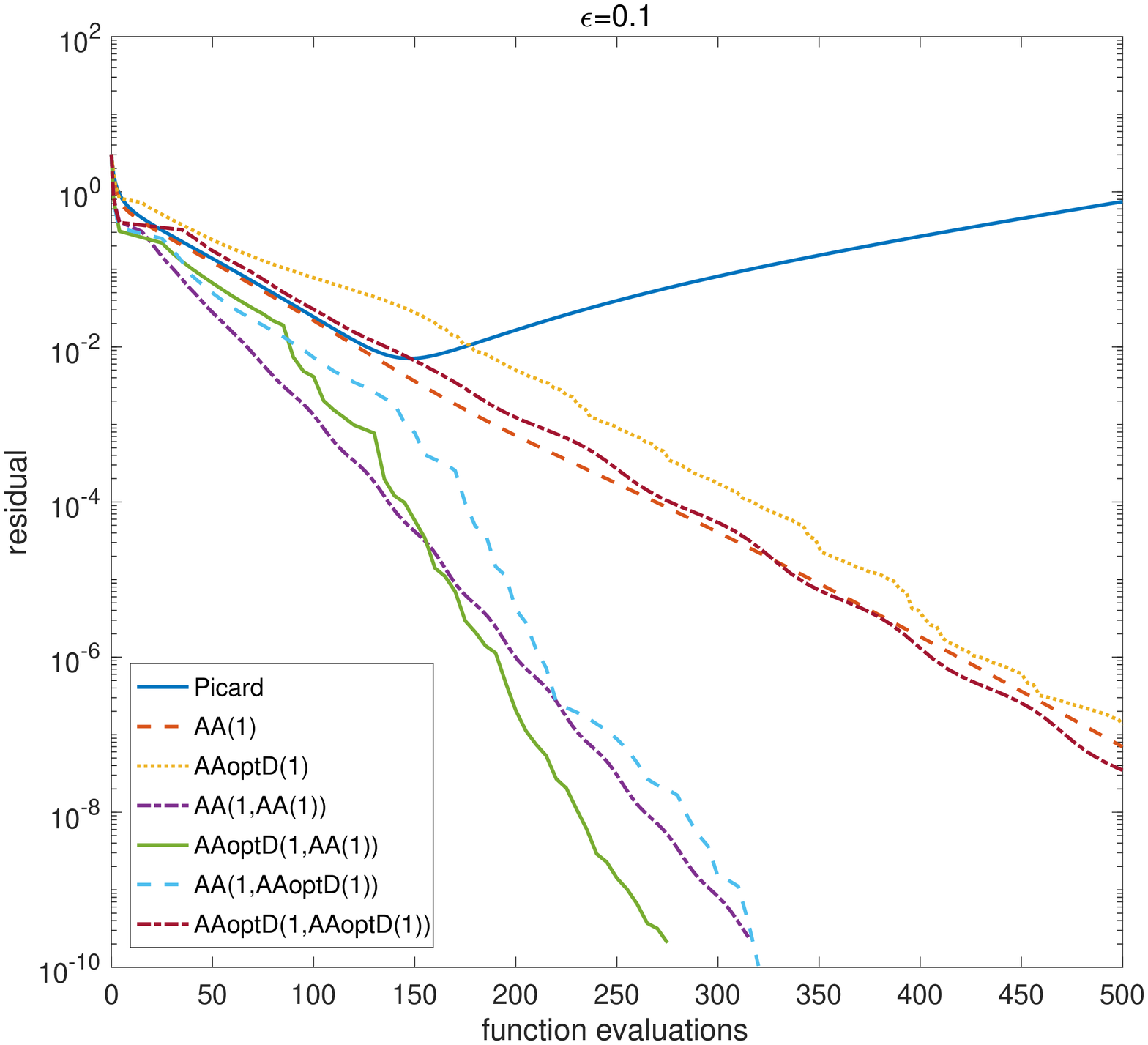}
  \caption{Solve the convection-diffusion problem: multiplicative composition with $\epsilon=0.1$.}
  \label{fig:fig_new61}
\end{figure}

\begin{figure}[htbp]
  \centering
  \includegraphics[height=3.8in, width=4.5in]{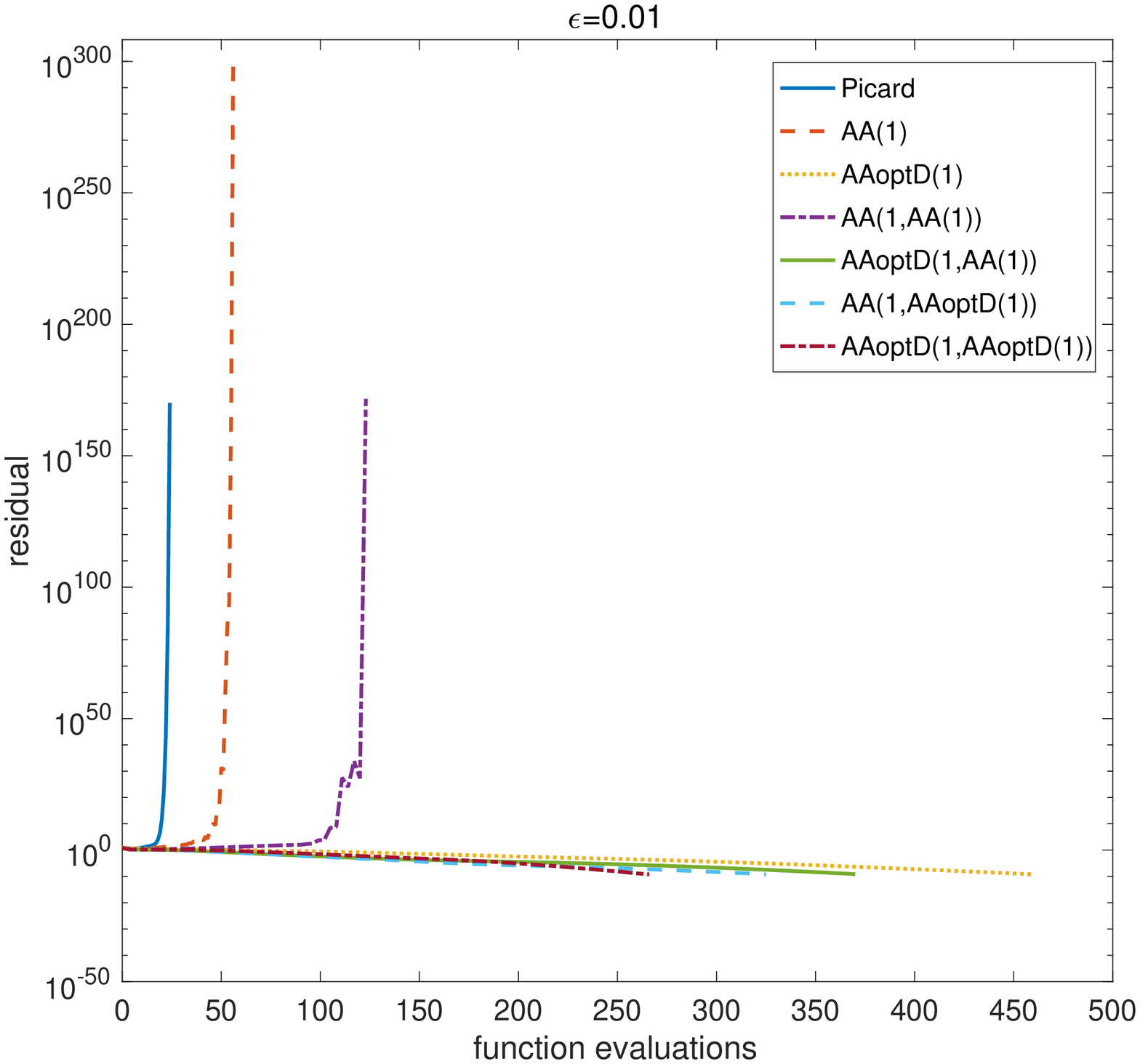}
  \caption{Solve the convection-diffusion problem: multiplicative composition with $\epsilon=0.01$.}
  \label{fig:fig_new62}
\end{figure}

\begin{figure}[htbp]
  \centering
  \includegraphics[height=3.8in, width=4.5in]{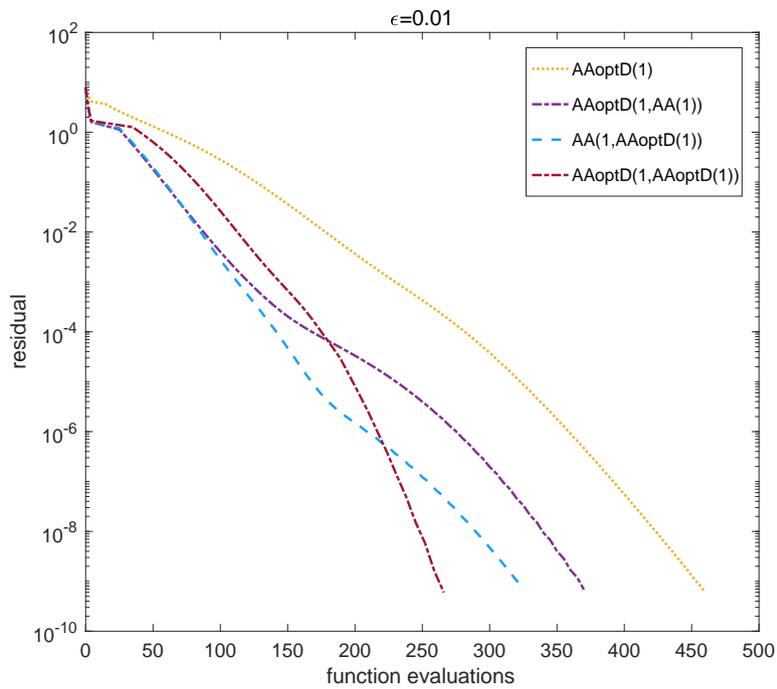}
  \caption{Solve the convection-diffusion problem with $\epsilon=0.01$ (zoom in).}
  \label{fig:fig_new63}
\end{figure}

\begin{figure}[htbp]
  \centering
  \includegraphics[height=3.8in, width=4.5in]{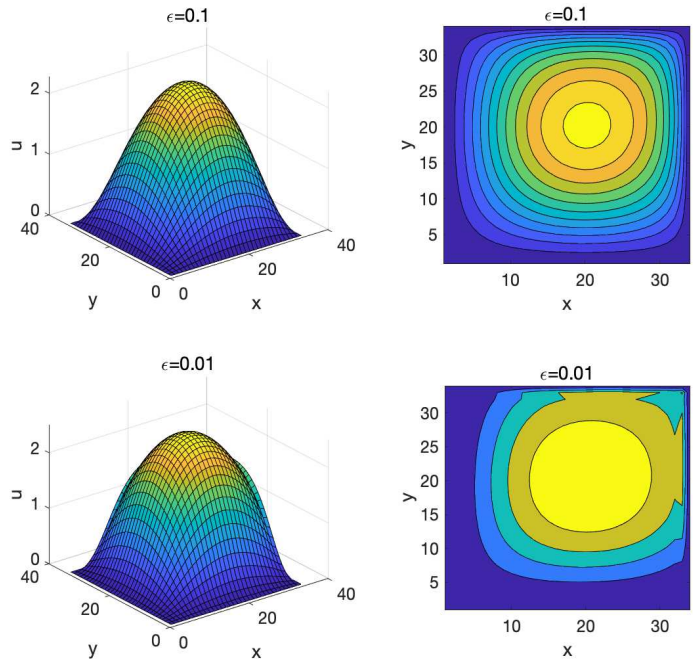}
  \caption{Solution the convection-diffusion problem with $\epsilon=0.1$ and $\epsilon=0.01$.}
  \label{fig:fig_new64}
\end{figure}

\begin{figure}[htbp]
  \centering
  \includegraphics[height=3.8in, width=4.5in]{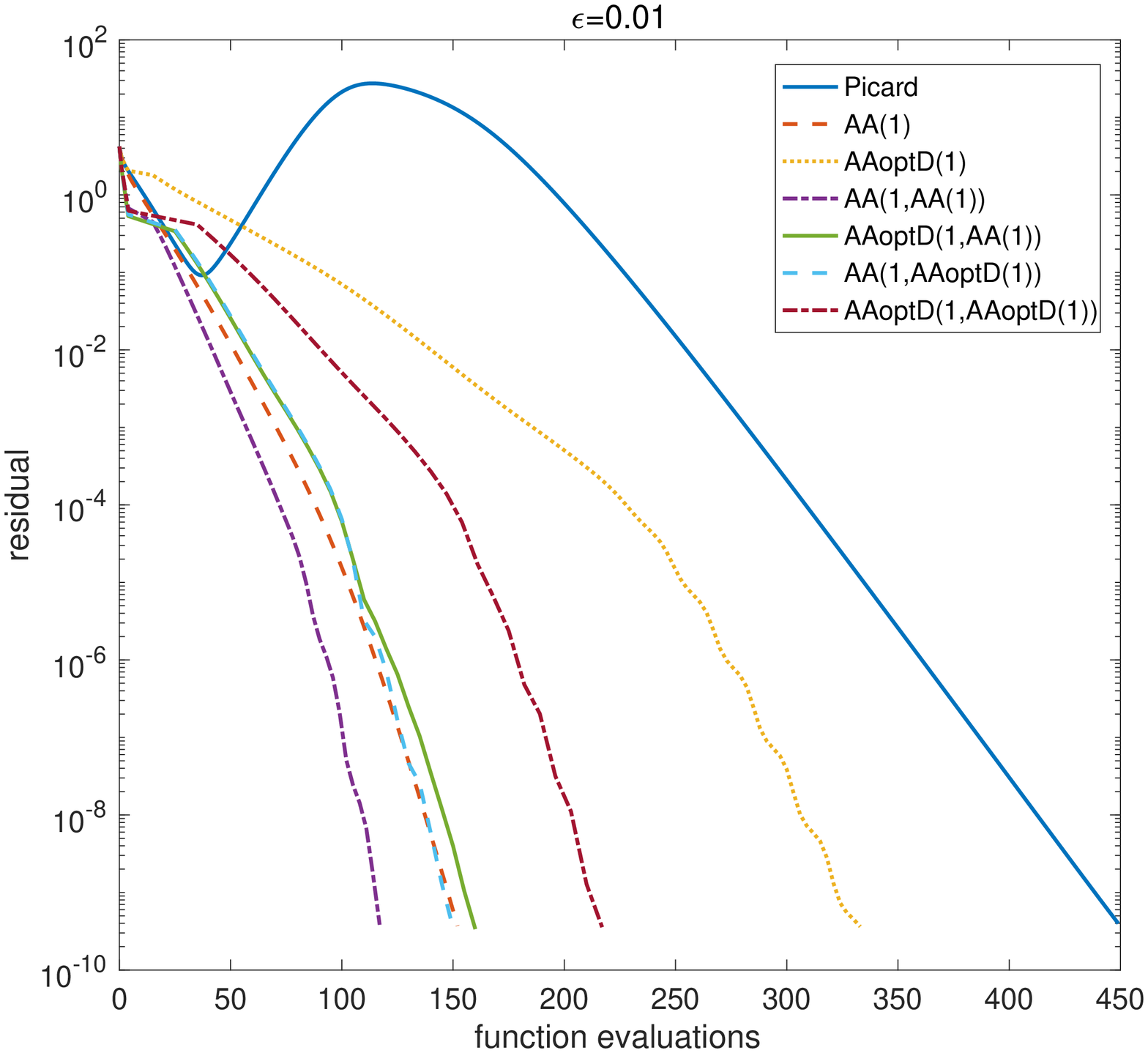}
  \caption{Solve the convection-diffusion problem: multiplicative composition with $\epsilon=0.01$.}
  \label{fig:fig_new621}
\end{figure}

\begin{figure}[htbp]
  \centering
  \includegraphics[height=2in, width=4.5in]{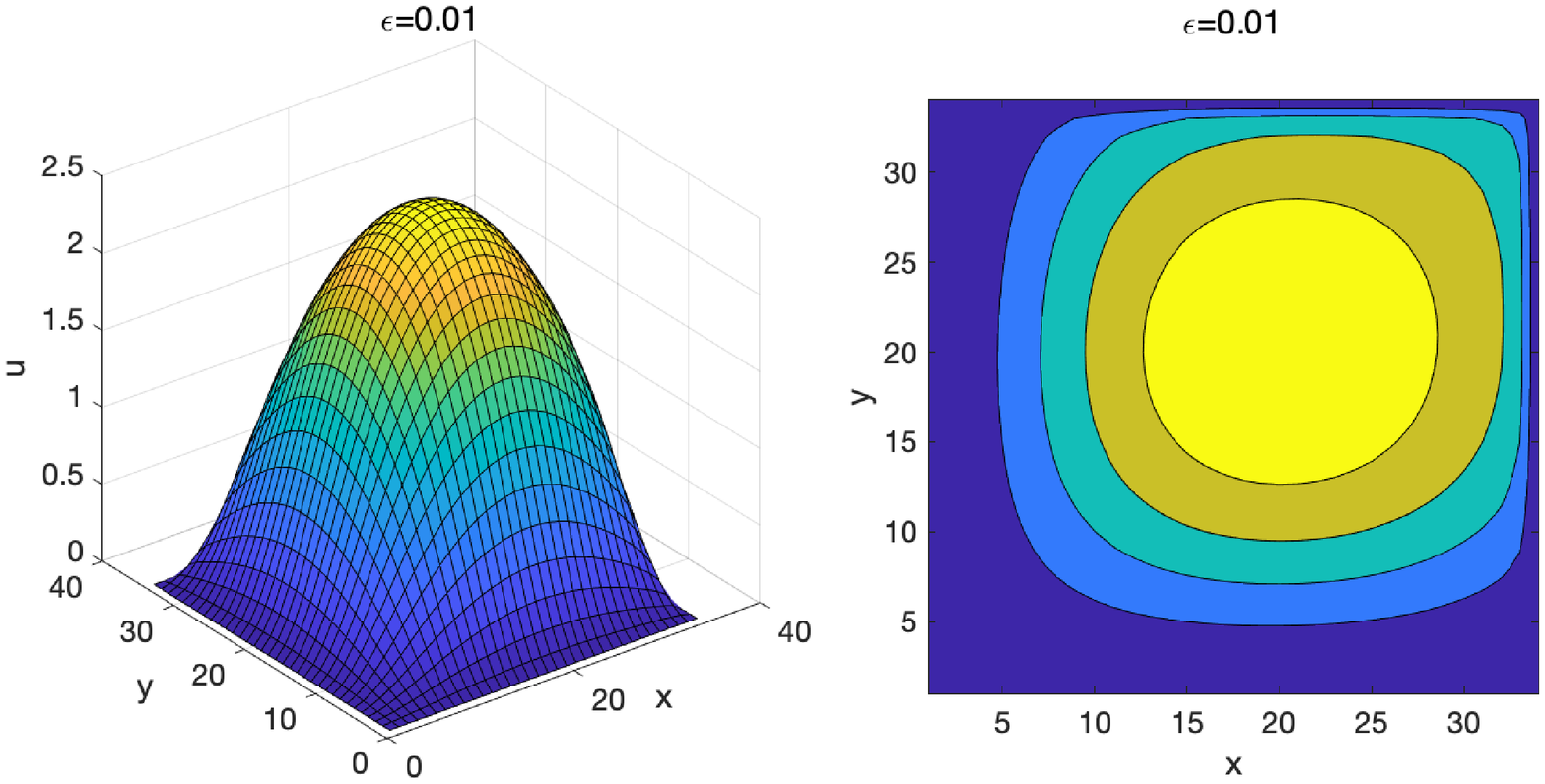}
  \caption{Solution the convection-diffusion problem: upwind scheme for convection term,  $\epsilon=0.01$.}
  \label{fig:fig_new641}
\end{figure}
% \begin{figure}[htbp]
%   \centering
%   \label{fig:f_new63}\includegraphics[height=3.6in, width=4.5in]{paper_convection_diffusion_ep_001.eps}
%   \caption{Solve the convection-diffusion problem with $\epsilon=0.01$ and a larger window size.}
%   \label{fig:fig_new63}
% \end{figure}

% %%%%%%%%%%%%%%%%%%% solve Ax=b
Our next example is about solving a linear system $Ax=b$. As proved by Walker and Ni in \cite{WaNi2011}, AA without truncation is ``essentially equivalent'' in a certain sense to the GMRES method for linear problems.
\begin{example} \textbf{The\ linear\ equations.} Apply AA and AAoptD to solve the following linear system $Ax=b$, where $A$ is
\begin{equation*}
A= 
\begin{pmatrix}
2 & -1 & \cdots & 0 & 0 \\
-1 & 2 & \cdots & 0 & 0 \\
\vdots  & \vdots  & \ddots & \vdots & \vdots \\
0 & 0 & \cdots & 2 &-1\\
0 & 0 & \cdots & -1 &2
\end{pmatrix},\ \ \ A\in R_{n\times n} 
\end{equation*}
and
$$b=(1,\cdots,1)^{T}.$$

Choose $n=100$ so that a large window size $m$ is needed in Anderson Acceleration. We also note here that the Picard iteration does not work for this problem. 
\end{example}

In our test, the initial guess is $x_0=(0,\cdots,0)^{T}$. The result is shown in \Cref{fig:fig_7}. From \Cref{fig:fig_7}, we see that the fully non-stationary Anderson acceleration methods work much better than the stationary AA method. This example also indicates that our proposed fully non-stationary AA can be also used to solve linear systems.
\begin{figure}[htbp]
  \centering
  \includegraphics[height=3.8in, width=4.5in]{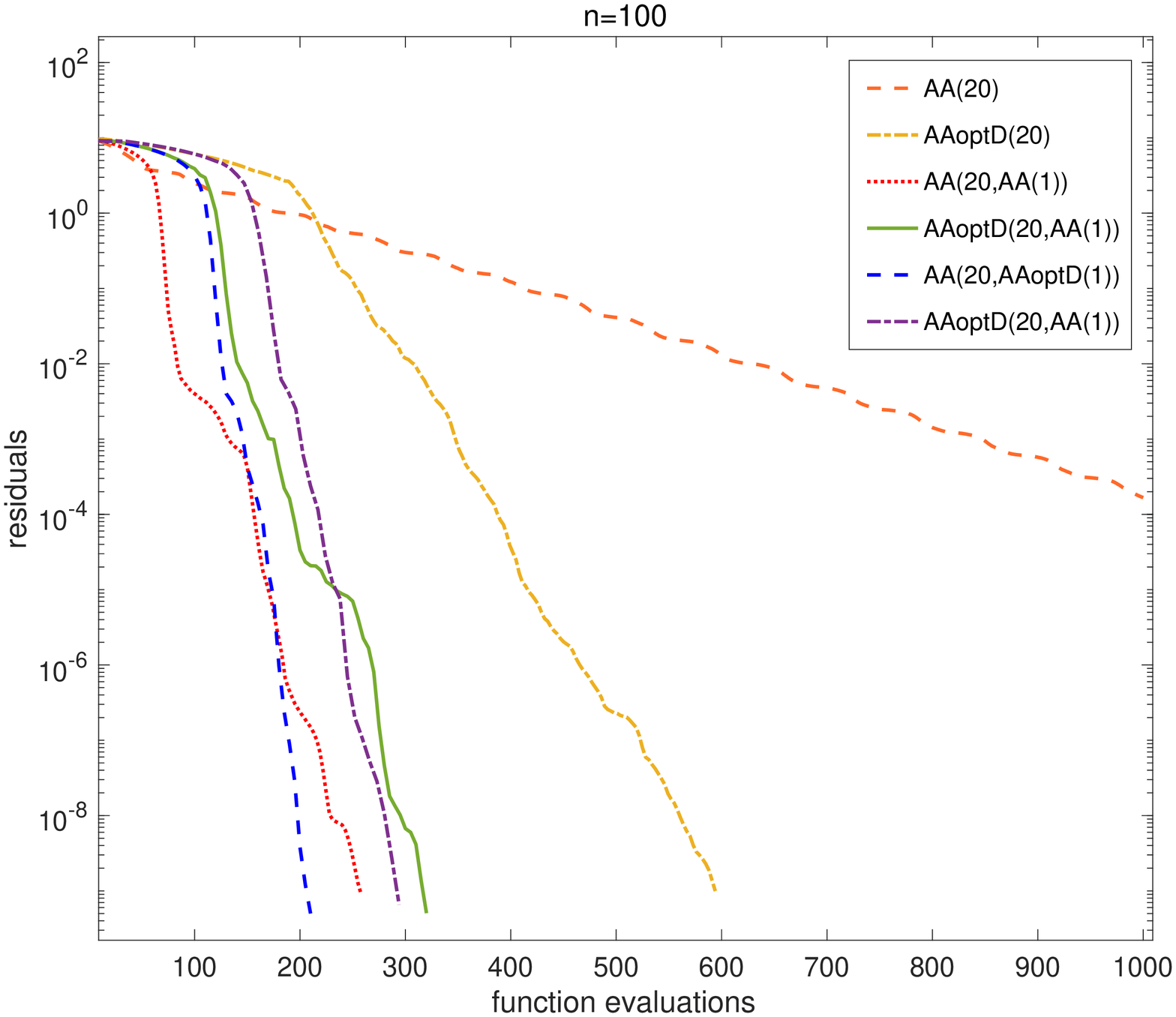}
  \caption{Non-stationary AA: solving a linear problem $Ax=b$ with $n=100$.}
  \label{fig:fig_7}
\end{figure}

\section{Conclusions}
\label{sec:conclusions}
In the present work, we propose and analyze a set of fully non-stationary Anderson acceleration algorithms with dynamic window sizes and optimized damping factors to further speed up linear and nonlinear iterations. In general, some of these non-stationary AA algorithms converge faster than the stationary AA method and they may significantly reduce the memory requirements and time to find the solution. For future guidance of choosing these non-stationary AA methods, our numerical results indicate that $AAoptD(m,AA(n))$ (with only one inner loop iteration and some very small inner window size $n$) usually converges much faster than stationary AA(m) and other non-stationary AA methods. This is not surprising since the local convergence rate of AA with damping factors is $\theta_k((1-\beta_{k-1})+\beta_{k-1}\kappa)$.\cite{Ev2020} However, two extra function evaluations are needed in each iteration. Therefore, if the function evaluation is not very expensive in related problems or when the storage of the computer is very limited, we suggest trying these types of fully non-stationary AA algorithms first. Otherwise, non-stationary AA methods like $AA(m,AA(n))$ (with only one inner loop iteration and some very small inner window size $n$) might be the good options since these algorithms usually converge faster than stationary $AA(m)$ and their total amount of work is less than that of applying $AA(m)$ twice. There is a lot of variety in these fully non-stationary Anderson acceleration methods, our future work will continue to explore the behaviors of these methods and test them on other broader problems.

\section*{Acknowledgments}
This work was partially supported by the National Natural Science Foundation of China [grant number 12001287]; the Startup Foundation for Introducing Talent of Nanjing University of Information Science and Technology [grant number 2019r106]; The first author Kewang Chen also gratefully acknowledge the financial support for his doctoral study provided by the China Scholarship Council (No. 202008320191).

\clearpage

\nocite{*}% Show all bib entries - both cited and uncited; comment this line to view only cited bib entries;
\bibliography{wileyNJD-AMA}

\begin{thebibliography}{10}
\providecommand \doibase [0]{http://dx.doi.org/}%

\bibitem{An1965}
Anderson DG. Iterative procedures for nonlinear integral equations. {\it J.
  Assoc. Comput. Mach.} 1965\string; 12\string: 547--560.
\newblock \href {\doibase 10.1145/321296.321305} {doi: 10.1145/321296.321305}

\bibitem{An2019}
Anderson DGM. Comments on ``{A}nderson acceleration, mixing and
  extrapolation''. {\it Numer. Algorithms} 2019\string; 80(1)\string: 135--234.
\newblock \href {\doibase 10.1007/s11075-018-0549-4} {doi:
  10.1007/s11075-018-0549-4}

\bibitem{ToKe2015}
Toth A, Kelley CT. Convergence analysis for {A}nderson acceleration. {\it SIAM
  J. Numer. Anal.} 2015\string; 53(2)\string: 805--819.
\newblock \href {\doibase 10.1137/130919398} {doi: 10.1137/130919398}

\bibitem{wa2011}
Walker HF. Anderson acceleration: Algorithms and implementations. {\it WPI
  Math. Sciences Dept. Report MS-6-15-50} 2011.

\bibitem{CaMi1998}
Carlson NN, Miller K. Design and application of a gradient-weighted moving
  finite element code. {I}. {I}n one dimension. {\it SIAM J. Sci. Comput.}
  1998\string; 19(3)\string: 728--765.
\newblock \href {\doibase 10.1137/S106482759426955X} {doi:
  10.1137/S106482759426955X}

\bibitem{Mi2005}
Miller K. Nonlinear {K}rylov and moving nodes in the method of lines. {\it J.
  Comput. Appl. Math.} 2005\string; 183(2)\string: 275--287.
\newblock \href {\doibase 10.1016/j.cam.2004.12.032} {doi:
  10.1016/j.cam.2004.12.032}

\bibitem{OoWa2000}
Oosterlee CW, Washio T. Krylov subspace acceleration of nonlinear multigrid
  with application to recirculating flows. In: . 21.  2000 (pp. 1670--1690)

\bibitem{WaOo1997}
Washio T, Oosterlee CW. Krylov subspace acceleration for nonlinear multigrid
  schemes. In: . 6.  1997 (pp. 271--290).

\bibitem{WaNi2011}
Walker HF, Ni P. Anderson acceleration for fixed-point iterations. {\it SIAM J.
  Numer. Anal.} 2011\string; 49(4)\string: 1715--1735.
\newblock \href {\doibase 10.1137/10078356X} {doi: 10.1137/10078356X}

\bibitem{LiYa2013}
Lin L, Yang C. Elliptic preconditioner for accelerating the self-consistent
  field iteration in {K}ohn-{S}ham density functional theory. {\it SIAM J. Sci.
  Comput.} 2013\string; 35(5)\string: S277--S298.
\newblock \href {\doibase 10.1137/120880604} {doi: 10.1137/120880604}

\bibitem{Pu1980}
Pulay P. Convergence acceleration of iterative sequences. the case of {SCF}
  iteration. {\it Chemical Physics Letters} 1980\string; 73(2)\string: 393-398.
\newblock \href {\doibase https://doi.org/10.1016/0009-2614(80)80396-4} {doi:
  https://doi.org/10.1016/0009-2614(80)80396-4}

\bibitem{Pu1982}
Pulay P. Improved {SCF} convergence acceleration. {\it Journal of Computational
  Chemistry} 1982\string; 3(4)\string: 556--560.
\newblock \href {\doibase 10.1002/jcc.540030413} {doi: 10.1002/jcc.540030413}

\bibitem{EiNe1987}
Eirola T, Nevanlinna O. Accelerating with rank-one updates. In: . 121.  1989
  (pp. 511--520)

\bibitem{Ey1996}
Eyert V. A comparative study on methods for convergence acceleration of
  iterative vector sequences. {\it J. Comput. Phys.} 1996\string;
  124(2)\string: 271--285.
\newblock \href {\doibase 10.1006/jcph.1996.0059} {doi: 10.1006/jcph.1996.0059}

\bibitem{FaSa2009}
Fang Hr, Saad Y. Two classes of multisecant methods for nonlinear acceleration.
  {\it Numer. Linear Algebra Appl.} 2009\string; 16(3)\string: 197--221.
\newblock \href {\doibase 10.1002/nla.617} {doi: 10.1002/nla.617}

\bibitem{Ha2010}
Haelterman R, Degroote J, Van~Heule D, Vierendeels J. On the similarities
  between the quasi-{N}ewton inverse least squares method and {GMRES}. {\it
  SIAM J. Numer. Anal.} 2010\string; 47(6)\string: 4660--4679.
\newblock \href {\doibase 10.1137/090750354} {doi: 10.1137/090750354}

\bibitem{Ya2009}
Yang C, Meza JC, Lee B, Wang LW. K{SSOLV}---a {MATLAB} toolbox for solving the
  {K}ohn-{S}ham equations. {\it ACM Trans. Math. Software} 2009\string;
  36(2)\string: Art. 10, 35.
\newblock \href {\doibase 10.1145/1499096.1499099} {doi:
  10.1145/1499096.1499099}

\bibitem{Ev2020}
Evans C, Pollock S, Rebholz LG, Xiao M. A proof that {A}nderson acceleration
  improves the convergence rate in linearly converging fixed-point methods (but
  not in those converging quadratically). {\it SIAM J. Numer. Anal.}
  2020\string; 58(1)\string: 788--810.
\newblock \href {\doibase 10.1137/19M1245384} {doi: 10.1137/19M1245384}

\bibitem{Po2019}
Pollock S, Rebholz LG, Xiao M. Anderson-accelerated convergence of {P}icard
  iterations for incompressible {N}avier-{S}tokes equations. {\it SIAM J.
  Numer. Anal.} 2019\string; 57(2)\string: 615--637.
\newblock \href {\doibase 10.1137/18M1206151} {doi: 10.1137/18M1206151}

\bibitem{DeHe2021}
De~Sterck H, He Y. On the asymptotic linear convergence speed of {A}nderson
  acceleration, {N}esterov acceleration, and nonlinear {GMRES}. {\it SIAM J.
  Sci. Comput.} 2021\string; 43(5)\string: S21--S46.
\newblock \href {\doibase 10.1137/20M1347139} {doi: 10.1137/20M1347139}

\bibitem{WaSt2021}
Wang D, He Y, De~Sterck H. On the asymptotic linear convergence speed of
  {A}nderson acceleration applied to {ADMM}. {\it J. Sci. Comput.} 2021\string;
  88(2)\string: Paper No. 38, 35.
\newblock \href {\doibase 10.1007/s10915-021-01548-2} {doi:
  10.1007/s10915-021-01548-2}

\bibitem{BiKe2021}
Bian W, Chen X, Kelley CT. Anderson acceleration for a class of nonsmooth
  fixed-point problems. {\it SIAM J. Sci. Comput.} 2021\string; 43(5)\string:
  S1--S20.
\newblock \href {\doibase 10.1137/20M132938X} {doi: 10.1137/20M132938X}

\bibitem{Br2015}
Brune PR, Knepley MG, Smith BF, Tu X. Composing scalable nonlinear algebraic
  solvers. {\it SIAM Rev.} 2015\string; 57(4)\string: 535--565.
\newblock \href {\doibase 10.1137/130936725} {doi: 10.1137/130936725}

\bibitem{YuLi2018}
Peng Y, Deng B, Zhang J, Geng F, Qin W, Liu L. Anderson acceleration for
  geometry optimization and physics simulation. {\it ACM Transactions on
  Graphics (TOG)} 2018\string; 37(4)\string: 1--14.
\newblock \href {\doibase 10.1145/3197517.3201290} {doi:
  10.1145/3197517.3201290}

\bibitem{WeGa2019}
Shi W, Song S, Wu H, Hsu YC, Wu C, Huang G. Regularized {A}nderson acceleration
  for off-policy deep reinforcement learning. {\it arXiv preprint
  arXiv:1909.03245} 2019.

\bibitem{To2017}
Toth A, Ellis JA, Evans T, et al. Local improvement results for {A}nderson
  acceleration with inaccurate function evaluations. {\it SIAM J. Sci. Comput.}
  2017\string; 39(5)\string: S47--S65.
\newblock \href {\doibase 10.1137/16M1080677} {doi: 10.1137/16M1080677}

\bibitem{Ya2021}
Yang Y. Anderson acceleration for seismic inversion. {\it Geophysics}
  2021\string; 86(1)\string: R99--R108.
\newblock \href {\doibase 10.1190/geo2020-0462.1} {doi: 10.1190/geo2020-0462.1}

\bibitem{Zh2020}
Zhang J, O'Donoghue B, Boyd S. Globally convergent type-{I} {A}nderson
  acceleration for nonsmooth fixed-point iterations. {\it SIAM J. Optim.}
  2020\string; 30(4)\string: 3170--3197.
\newblock \href {\doibase 10.1137/18M1232772} {doi: 10.1137/18M1232772}

\bibitem{Chen2022}
Chen K, Vuik C. Non-stationary Anderson acceleration with optimized damping.
  {\it arXiv preprint arXiv:2202.05295} 2022.
\newblock \href {\doibase 10.48550/arXiv.2202.05295} {doi:
  10.48550/arXiv.2202.05295}

\bibitem{Brown2012}
Brown J, Knepley MG, May DA, McInnes LC, Smith B. Composable linear solvers for
  multiphysics. In: IEEE. ; 2012\string: 55--62

\bibitem{Kirby2018}
Kirby RC, Mitchell L. Solver composition across the PDE/linear algebra barrier.
  {\it SIAM Journal on Scientific Computing} 2018\string; 40(1)\string:
  C76--C98.
\newblock \href {\doibase 10.1137/17M1133208} {doi: 10.1137/17M1133208}

\bibitem{vvuik94}
Vorst v.~dHA, Vuik C. G{MRESR}: a family of nested {GMRES} methods. {\it Numer.
  Linear Algebra Appl.} 1994\string; 1(4)\string: 369--386.
\newblock \href {\doibase 10.1002/nla.1680010404} {doi: 10.1002/nla.1680010404}

\bibitem{vuik93}
Vuik C. Solution of the discretized incompressible Navier-Stokes equations with
  the GMRES method. {\it International Journal for Numerical Methods in Fluids}
  1993\string; 16(6)\string: 507--523.
\newblock \href {\doibase 10.1002/fld.1650160605} {doi: 10.1002/fld.1650160605}

\bibitem{Gl1985}
Glowinski R, Keller HB, Reinhart L. Continuation-conjugate gradient methods for
  the least squares solution of nonlinear boundary value problems. {\it SIAM J.
  Sci. Statist. Comput.} 1985\string; 6(4)\string: 793--832.
\newblock \href {\doibase 10.1137/0906055} {doi: 10.1137/0906055}

\bibitem{Pe1998}
Pernice M, Walker HF. N{ITSOL}: a {N}ewton iterative solver for nonlinear
  systems. In: . 19.  1998 (pp. 302--318)

\bibitem{De2012}
De~Sterck H. A nonlinear {GMRES} optimization algorithm for canonical tensor
  decomposition. {\it SIAM J. Sci. Comput.} 2012\string; 34(3)\string:
  A1351--A1379.
\newblock \href {\doibase 10.1137/110835530} {doi: 10.1137/110835530}

\bibitem{Ke1978}
Kelley CT, Mullikin TW. Solution by iteration of {$H$}-equations in multigroup
  neutron transport. {\it J. Mathematical Phys.} 1978\string; 19(2)\string:
  500--501.
\newblock \href {\doibase 10.1063/1.523673} {doi: 10.1063/1.523673}

\bibitem{RoSc2011}
Rohwedder T, Schneider R. An analysis for the {DIIS} acceleration method used
  in quantum chemistry calculations. {\it J. Math. Chem.} 2011\string;
  49(9)\string: 1889--1914.
\newblock \href {\doibase 10.1007/s10910-011-9863-y} {doi:
  10.1007/s10910-011-9863-y}

\end{thebibliography}

\end{document}